\documentclass[10pt]{article}
\usepackage{color}
\usepackage{amssymb}
\usepackage{amsthm,array,amsfonts}
\usepackage{amsmath}
\usepackage{verbatim}
\usepackage{tikz-cd}
\usepackage{enumitem}
\setitemize{itemsep=0pt}
\setenumerate{itemsep=0pt}

\usepackage{stmaryrd} %fatslash

\usepackage{hyperref}
\hypersetup{
	colorlinks,
	citecolor=blue,
	filecolor=black,
	linkcolor=black,
}

\newtheorem{theo}{Theorem}[section]
\newtheorem{thm}{Theorem}[section]
\newtheorem{prop}[theo]{Proposition}

\newtheorem{lemma}[theo]{Lemma}

\newtheorem{problem}[theo]{Problem}

\newtheorem{conj}[theo]{Conjecture}

\theoremstyle{remark}

\newtheorem{rmk}[theo]{Remark}

\newcommand{\BA}{{\mathbb{A}}}

\newcommand{\BC}{{\mathbb{C}}}

\newcommand{\BF}{{\mathbb{F}}}
\newcommand{\BG}{{\mathbb{G}}}

\newcommand{\BI}{{\mathbb{I}}}

\newcommand{\BQ}{{\mathbb{Q}}}
\newcommand{\BR}{{\mathbb{R}}}

\newcommand{\BZ}{{\mathbb{Z}}}

\newcommand{\CA}{{\mathcal A}}
\newcommand{\CB}{{\mathcal B}}

\newcommand{\CD}{{\mathcal D}}
\newcommand{\CE}{{\mathcal E}}
\newcommand{\CF}{{\mathcal F}}

\newcommand{\CI}{{\mathcal I}}
\newcommand{\CJ}{{\mathcal J}}
\newcommand{\CK}{{\mathcal K}}
\newcommand{\CL}{{\mathcal L}}

\newcommand{\CO}{{\mathcal O}}

\newcommand{\CQ}{{\mathcal Q}}

\newcommand{\CS}{{\mathcal S}}

\newcommand{\CZ}{{\mathcal Z}}

\newcommand{\Fq}{{\mathfrak{q}}}

\newcommand{\Ft}{{\mathfrak{t}}}

\newcommand{\ch}{\mathsf{ch}}

\newcommand{\id}{\mathrm{id}}

\newcommand{\pt}{{\mathsf{p}}}
\newcommand{\td}{{\mathrm{td}}}
\newcommand{\blangle}{\big\langle}
\newcommand{\brangle}{\big\rangle}

\newcommand{\Mbar}{{\overline M}}
\newcommand\ev{\operatorname{ev}}
\newcommand{\Pic}{\mathop{\rm Pic}\nolimits}

\newcommand{\DT}{\mathsf{DT}}
\newcommand{\GW}{\mathsf{GW}}

\DeclareFontFamily{OT1}{rsfs}{}
\DeclareFontShape{OT1}{rsfs}{n}{it}{<-> rsfs10}{}
\DeclareMathAlphabet{\curly}{OT1}{rsfs}{n}{it}
\renewcommand\hom{\curly H\!om}

\newcommand\Ext{\operatorname{Ext}}
\newcommand\Hom{\operatorname{Hom}}
\newcommand{\p}{\mathbb{P}}

\newcommand\Spec{\operatorname{Spec}}
\newcommand\Quot{\operatorname{Quot}}
\newcommand\Hilb{\operatorname{Hilb}}

\newcommand\Coh{\operatorname{Coh}}

\newcommand{\vd}{\mathsf{vd}}
\newcommand{\vir}{\mathsf{vir}}

\newcommand{\std}{\mathsf{std}}
\newcommand{\E}{E}
\newcommand{\HH}{\mathsf{H}}

\newcommand{\K}{\mathsf{K3}}

\newcommand{\rk}{\operatorname{\mathsf{rk}}}

\newcommand{\Obs}{\mathrm{Obs}}

\newcommand\Cone{\operatorname{Cone}}
\renewcommand\div{\operatorname{div}}
\title{Multiple cover formulas for K3 geometries, wall-crossing, and Quot schemes}
\author{Georg Oberdieck}
\date{}

\begin{document}
	\maketitle
	\setcounter{section}{0}
	
	%\vspace{-20pt}
	\begin{abstract}
		Let $S$ be a K3 surface. We study the reduced Donaldson-Thomas theory of the cap $(S \times \p^1) / S_{\infty}$ by a second cosection argument. We obtain four main results: \\
		(i) A multiple cover formula for the rank 1 Donaldson-Thomas theory of $S \times E$, leading to a complete solution of this theory. 
		(ii) Evaluation of the wall-crossing term in Nesterov's quasimap wall-crossing between the punctual Hilbert schemes and Donaldson-Thomas theory of $S \times \text{Curve}$.
		(iii) A multiple cover formula for the genus $0$ Gromov-Witten theory of punctual Hilbert schemes.
		(iv) Explicit evaluations of virtual Euler numbers of Quot schemes of stable sheaves on K3 surfaces.
	\end{abstract}
	%\vspace{-10pt}
	%\setcounter{tocdepth}{2}
	%\tableofcontents

	\section{Introduction}
	\subsection{Overview}
	Let $S$ be a K3 surface.
	In this paper we consider three different types of counting theories:
	\begin{itemize}[itemsep=0pt]
		\item Gromov-Witten theory of moduli spaces of stable sheaves on $S$,
		\item Donaldson-Thomas theory of $S \times E$, where
		$E$ is an elliptic curve,
		\item Virtual Euler characteristics of Quot schemes of stable sheaves.
	\end{itemize}
	As usual for K3 geometries, in all three theories the standard virtual fundamental classes of the moduli spaces vanish. Instead the counting theories are defined by a reduced virtual class.
	This deviation from the standard theory
	leads to surprising additional structure of the invariants.
	Two of them are taken up in this work.
	First, one expects a multiple cover formula that expresses counts in imprimitive (curve) classes in terms of those for primitive classes \cite{KKV, KMPS, PT, K3xE, ObMC}.
	In this paper we prove such multiple cover formulas for the genus $0$ Gromov-Witten theory of the punctual Hilbert schemes and the rank $1$ Donaldson-Thomas theory of $S \times \E$. In particular, together with the results of \cite{HAE} the latter determines all rank $1$ reduced Donaldson-Thomas invariants of the (non-strict) Calabi-Yau threefold $S \times \E$.
	The second structural result concerns wall-crossing formulas, where due to $\epsilon$-calculus \cite{ObShen, OPT}, one expects the vanishing of almost all wall-crossing contributions.
	We will consider the case of Nesterov's quasimap wall-crossing
	between moduli spaces of sheaves on K3 surfaces, in particular the punctual Hilbert schemes, and Donaldson-Thomas theory. Nesterov shows that these theories are related by a single wall-crossing term.
	We prove that this single term is precisely the virtual Euler number of the Quot scheme. We then use this connection to determine these virtual Euler numbers explicitly. The outcome is an intimate connection between all three counting theories above.
	
	\subsection{Three theories}
	\subsubsection{Punctual Hilbert schemes}
	For the first geometry, consider the Hilbert scheme $S^{[n]}$ of $n$ points on the K3 surface $S$.
	%Let $S^{[n]}$ be the Hilbert scheme of $n$ points on $S$.
	There exists a canonical isomorphism
	\[ H_2(S^{[n]}, \BZ) \cong H_2(S,\BZ) \oplus \BZ A \]
	where $A$ is the extremal curve of the Hilbert-Chow morphism $S^{[n]} \to \mathrm{Sym}^{n}(S)$, \cite[Sec.1.2]{HilbK3}.
	%where $A$ is the class of the fiber of the Hilbert-Chow morphism $S^{[n]} \to \mathrm{Sym}^{n}(S)$ over a general point of the singular locus,
	%see \cite[Sec.1.2]{HilbK3}.
	
	Let $\beta \in H_2(S,\BZ)$ be an effective curve class,
	and let $\Mbar_{E}(S^{[n]}, \beta+mA)$ be the 
	moduli space of (unmarked) degree $\beta + mA$ stable maps from nodal degenerations of the elliptic curve $E$ to $S^{[n]}$, see \cite{K3xE, NO, Pan} for details.
	The moduli space is of reduced\footnote{We will denote the reduced virtual fundamental classes in this paper by $[ - ]^{\vir}$.
		The (almost always vanishing) ordinary virtual class associated to the standard perfect obstruction theory will be denoted by $[ - ]^{\std}$.} virtual dimension $0$.
	We define
	the Gromov-Witten count of elliptic curves in $S^{[n]}$ of class $\beta+mA$ with fixed $j$-invariant by
	\begin{equation} \label{GWE count}
		\GW^{S^{[n]}}_{E, \beta,m}
		%\HH_{n,\beta, m}
		=
		\int_{[ \Mbar_{E}(S^{[n]}, \beta+mA) ]^{\vir}} 1. \end{equation}
	%The number $\GW^{S^{[n]}}_{E, \beta,m}$
	%is the Gromov-Witten count of elliptic curves in $S^{[n]}$ of class $\beta+mA$ with fixed $j$-invariant.
	The invariant $\GW^{S^{[n]}}_{E, \beta,m}$ is related to the actual enumerative count of elliptic curves with fixed $j$-invariant
	by a (conjectural) genus $0$ correction term, see \cite{NO}.

	\subsubsection{{$S \times \E$}}
	In this second geometry we consider the Hilbert scheme of curves in $S \times E$:
	\[ \Hilb_m(S \times E, (\beta,n)) = \{ Z \subset S \times E \, | \, 
	\ch_3(\CO_Z) = m, [Z] = (\beta,n) \}, \]
	where we use the identification given by the K\"unneth decomposition
	\[ H_2(S \times E, \BZ) \cong H_2(S,\BZ) \oplus \BZ E. \]
	The elliptic curve (viewed as a group) acts on the Hilbert scheme by translation. The stack quotient
	$\Hilb_m(S \times E, (\beta,n)) / E$ is of reduced virtual dimension $0$
	(in fact, an \'etale cover carries a symmetric perfect obstruction theory \cite{ObReduced}), so we can define:
	\[
	\DT^{S \times E}_{m, (\beta,n)}
	=
	\int_{ [ \Hilb_m(S \times E, (\beta,n)) / E ]^{\text{vir}} } 1.
	\]
	The number $\DT_{n, \beta,m} \in \BQ$ is the Donaldson-Thomas count of curves in $S \times E$ in class
	$(\beta,n[\p^1])$ up to translation.
	
	\subsubsection{Quot schemes}
	Out third geometries are the Quot schemes.
	Let $F \in \Coh(S)$ be a coherent sheaf of positive rank which is Gieseker stable with respect to some polarization. We consider the Quot scheme
	\[ \Quot(F,u) = \{ F \twoheadrightarrow Q \, |\, v(Q) = u \}\]
	where $v(Q) := \ch(Q) \sqrt{\td_S}$ is the Mukai vector.
	The moduli space has a reduced perfect obstruction theory with
	virtual tangent bundle $T^{\vir} = R \Hom_S(\CK, \CQ) + \CO$,
	see Section~\ref{subsec:quot k3}.
	The virtual Euler characteristic of the moduli space
	is defined after Fantechi and G\"ottsche \cite{FG} to be:
	\[
	e^{\vir}( \Quot(F,u) )
	=
	\int_{[ \Quot(F,u) ]^{\text{vir}}} c_{\vd}( T^{\vir} )
	\]
	where $\vd = \rk( T^{\vir} )$ is the virtual dimension.
	If $F$ is the structure sheaf these Euler characteristics
	were studied by Oprea and Pandharipande in \cite{OP}
	and are related to the Kawai-Yoshioka formula \cite{KY}.
	If $F = I_{\eta}$ is the ideal sheaf of a length $n$ subscheme $\eta \in S^{[n]}$, and $u = (0,\beta,m)$ we write
	\[ \mathsf{Q}_{n,(\beta,m)} := 
	e^{\vir}\left( \Quot(I_{\eta},(0,\beta,m)) \right). \]

	\subsubsection{Wall-crossing}
	Denis Nesterov in \cite{N1, N2} uses quasimaps to prove that
	\[
	\DT^{S \times E}_{m, (\beta,n)}
	=
	\GW^{S^{[n]}}_{E, \beta,m}
	+ (\text{Wall-crossing correction}),
	\]
	where the wall-crossing term is the
	contribution from the extremal component
	in the relative Donaldson-Thomas theory of $S \times \p^1 / S_{\infty}$.
	Our first main result is to make this wall-crossing term more explicit and relate it to the Quot scheme.
	\begin{thm} \label{thm:Wall  cross rank 1}
		\[
		\DT^{S \times E}_{m, (\beta,n)}
		=
		\GW^{S^{[n]}}_{E, \beta,m}
		-
		\chi(S^{[n]})
		\sum_{r | (\beta,m)} \frac{1}{r} (-1)^{m}
		\mathsf{Q}_{n,\frac{1}{r}(\beta,m)}
		\]
	\end{thm}
	
	Our second result is a complete evaluation of the wall-crossing term:
	\begin{thm} \label{thm:Correction term}
		The invariant $\mathsf{Q}_{n,(\beta,m)}$ only depends on the square $\beta \cdot \beta = 2h-2$.
		Moreover, if we write $\mathsf{Q}_{n,h,m} := \mathsf{Q}_{n,(\beta,m)}$ for this value, we have
		\[
		\sum_{h \geq 0} \sum_{m \in \BZ}
		\mathsf{Q}_{n,h,m} p^m q^{h-1}
		=
		\frac{\mathbf{G}(p,q)^{n}}{ \Theta(p,q)^2 \Delta(q) }
		\]
		where we let
		\begin{gather*}
			\Theta(p,q)  
			%= \frac{\sum_{\nu\in \mathbb{Z}+\frac{1}{2}} (-1)^{\lfloor \nu \rfloor} p^\nu q^{\nu^2/2}}{\eta(\tau)^3} 
			=  (p^{1/2}-p^{-1/2})\prod_{m\geq 1} \frac{(1-pq^m)(1-p^{-1}q^m)}{(1-q^m)^{2}} \\
			\Delta(\tau) = q \prod_{n \geq 1} (1-q^n)^{24}
		\end{gather*}
		and $\mathbf{G}(p,q) = -\Theta(p,q)^2 \left( p \frac{d}{dp} \right)^2 \log(\Theta(p,q))$.
	\end{thm}
	
	\begin{rmk}
		The case $n=0$ has been obtained previously in \cite[Thm 21]{OP}.
	\end{rmk}
	\begin{rmk}
		Let $\beta_h \in \Pic(S)$ be a primitive effective class of square $2h-2$.
		Consider the generating series
		\begin{equation} \label{DTnHHn}
			\begin{gathered}
				\DT_n(p,q) = \sum_{h \geq 0} \sum_{m \in \BZ} \DT^{S \times E}_{m, (\beta_h,n} q^{h-1} (-p)^m \\
				\HH_n(p,q) = \sum_{h \geq 0} \sum_{m \in \BZ} \GW^{S^{[n]}}_{E, \beta_h+mA} q^{h-1} (-p)^m
			\end{gathered}
		\end{equation}
		The following evaluation was proven in \cite{HAE}:
		\begin{equation} \label{igusa evaluation}
			\sum_{n \geq 0} \DT_n(p,q) \tilde{q}^{n-1} = - \frac{1}{\chi_{10}(p,q, \tilde{q})}
		\end{equation}
		where $\chi_{10}$ is the Igusa cusp form (as in \cite[Eqn. (12)]{K3xE}).
		We obtain the complete evaluation
		\[
		\sum_{n \geq 0} \HH_{n}(p,q) \tilde{q}^{h-1}
		=
		- \frac{1}{\chi_{10}(p,q, \tilde{q})} +
		\frac{1}{\Theta^2 \Delta}
		\frac{1}{\tilde{q}} \prod_{n \geq 1} \frac{1}{( 1 - ( \tilde{q} \mathbf{G} )^n )^{24}}
		\]
		which was conjectured in \cite[Conj.A]{K3xE}.
	\end{rmk}
	
	%\begin{rmk}
	%Since we have a DT/PT correspondence for $K3 \times E$,
	%the wall-crossing correction for Hilb/PT is the same as the Hilb/DT wallcrossing.
	%Both are given by the difference of $1/\chi_{10}$ with the Hilb series.
	%\end{rmk}

	\subsection{Multiple cover formula: $S \times E$}
	The main tool we will employ in this paper is a second-cosection argument for the Donaldson-Thomas theory of the cap $(S \times \p^1) / S_{\infty}$.
	It will also imply the following multiple cover formula,
	conjectured on the Gromov-Witten side in \cite[Conj. B]{K3xE}.
	
	Define the (Fourier) coefficients of the Igusa cusp form by
	\[ c(h,n,m) := \left[ \frac{1}{\chi_{10}(p,q,\tilde{q})} \right]_{q^{h} \tilde{q}^{n} p^m}. \]
	By the evaluation \eqref{igusa evaluation} it is given up to a (confusing) index shift by the Donaldson-Thomas theory of $S \times E$
	for a primitive class $\beta_h$ of square $2h-2$,
	\[  \DT^{S \times E}_{m,(\beta_h,n)} = (-1)^{m+1} c( h-1, n-1, m) \]
	
	\begin{thm} \label{thm: mc K3xE}
		For any effective class $\beta \in \Pic(S)$,
		\[ \DT^{S \times E}_{m,(\beta,n)} = (-1)^{m+1} \sum_{r| (m,\beta)} \frac{1}{r} c\left( \frac{1}{2} (\beta/r)^2,\, n-1,\, m/r \right) \]
	\end{thm}

	\subsection{Multiple cover formula: Hilb}
	We consider Gromov-Witten invariants of the Hilbert scheme $S^{[n]}$ more generally. For classes $\gamma_1, \ldots, \gamma_N \in H^{\ast}(S^{[n]})$ and a tautological class $\alpha \in H^{\ast}(\Mbar_{g,N})$ they are defined by
	\[
	\blangle \alpha ; \gamma_1, \ldots \gamma_N \brangle^{S^{[n]}}_{g, \beta+mA}
	=
	\int_{[ \Mbar_{g,n}(S^{[n]}, \beta+mA) ]}
	\pi^{\ast}(\alpha) 
	\prod_i \ev_i^{\ast}(\gamma_i).
	\]
	where the integral is over the reduced virtual class
	and $\pi$ is the forgetful morphism to the moduli space $\Mbar_{g,N}$ of stable curves.
	A general multiple cover formula for these counts
	was conjectured in \cite[Conj.B]{ObMC}. We state an equivalent special case of the conjecture here:
	(The special case is equivalent to the general case by the deformation theory of hyperk\"ahler varieties, see \cite[Lemma.3]{ObMC}.)
	For every divisor $r|\beta$ let $S_r$ be a K3 surface and let
	\[ \varphi_r : H^{2}(S,\BR) \to H^{2}(S_r, \BR) \]
	be a {\em real} isometry such that $\varphi_r(\beta/r) \in H_2(S_r,\BZ)$
	is a primitive effective curve class.
	We extend $\varphi_r$ to the full cohomology lattice
	by $\varphi_r(\pt) = \pt$ and $\varphi_r(1) = 1$,
	where $\pt \in H^4(S,\BZ)$ is the class of a point.
	By acting factorwise in the Nakajima operators,
	the isometry $\varphi_r$ then naturally induces an isomorphism on the cohomology of the Hilbert schemes (see \eqref{induced on Hilb} below for details):
	\[
	\varphi_r : H^{\ast}(S^{[n]}) \to H^{\ast}( S_r^{[n]} ).
	\]
	
	\begin{conj} \label{conj mc hilb} We have
		\[
		\blangle \alpha ; \gamma_1, \ldots \gamma_N \brangle^{S^{[n]}}_{g, \beta+mA} 
		= \sum_{r | (\beta,m)}  r^{3g-3+N-\deg(\alpha)} (-1)^{m + \frac{m}{r}}
		\blangle \alpha ; \varphi_r(\gamma_1), \ldots \varphi_r(\gamma_N) \brangle^{S^{[n]}}_{g, \varphi_r\left( \frac{\beta}{r} \right)+\frac{m}{r} A}.
		\]
	\end{conj}

	Our main result here is the following:
	\begin{thm} \label{thm: mc Hilb}
		Conjecture~\ref{conj mc hilb} holds for $g=0$ and $N \leq 3$.
	\end{thm}
	
	In particular, the theorem expresses the structure constants of the reduced quantum cohomology of $S^{[n]}$ for arbitrary degree in terms of those for primitive degree.
	Moreover, the counts $\GW^{E}_{E,\beta,m}$ in \eqref{GWE count}
	can be computed in
	terms of genus $0$ invariants (by degenerating $E$),
	and so Theorem~\ref{thm: mc Hilb} also implies a multiple cover formula for these types of invariants.
	Since the Gromov-Witten theory of $S^{[n]}$ vanishes for $g > 1$ if $n \geq 3$ (and for $g > 2$ if $n=2$) by dimension reasons, this proves a large chunk of the general conjecture.
	
	Theorem~\ref{thm: mc Hilb} also gives a new proof of the classical Yau-Zaslow formula governing genus $0$ counts on K3 surfaces.
	The previous proofs given in \cite{KMPS} and \cite{PT}
	both relied on the Gromov-Witten/Noether-Lefschetz correspondence while ours does not.

	\subsection{Higher rank}
	Nesterov's wall-crossing also applies to moduli spaces of higher rank sheaves. Consider the lattice $\Lambda = H^{\ast}(S,\BZ)$ endowed with the Mukai pairing
	\[ (x \cdot y) := - \int_S x^{\vee} y, \]
	where, if we decompose an element $x \in \Lambda$ according to degree as $(\rho,D,n)$, we have written $x^{\vee} = (\rho,-D,n)$.
	Let $M(v)$ be a moduli space of Gieseker stable sheaves $F$ on $S$ of positive rank and Mukai vector $v(F) = \ch(F) \sqrt{\td_S} = v$.
	%Here stability is taken in the Gieseker sense with respect to some fixed polarization.
	We assume that stability and semi-stability agrees for sheaves in class $v$, so that $M(v)$ is proper. Assume also that there exists an algebraic class $y \in K_{\mathrm{alg}}(S)$ such that $v \cdot v(y) = 1$,
	which implies that $M(v)$ is fine.\footnote{We expect that this condition can be removed eventually.}
	%(with respect to some polarization)
	We refer to \cite[Sec.6]{HL} for the construction and the properties of $M(v)$.
	By work of Mukai there exists a canonical isomorphism
	(see Section~\ref{section:degree}):
	\[ \theta: (v^{\perp})^{\vee} \xrightarrow{\ \ \cong\ \ } H_2(M(v),\BZ). \]
	For $w' \in H_2(M(v), \BZ)$ we define parallel to before:
	\[
	\GW^{M(v)}_{E, w'}
	=
	\int_{[ \Mbar_{E}(M(v), w') ]^{\vir}} 1
	\]
	
	Following \cite[Section 3]{N1}, let also $M_{v,w}(S \times E)$ be the moduli space parametrizing torsion free sheaves $G$ of fixed determinant on $S \times E$ whose restriction to the generic fiber over the elliptic curve is Gieseker-stable, and which have Mukai vector
	\[ \ch(G) \sqrt{\td_S} = v + w \cdot \omega. \]
	Here $\omega \in H^2(E,\BZ)$ is the point class and we have suppressed pullbacks by the projections to $S$ and $E$. 
	We assume that $w \neq 0$ and define the counts:
	\[
	\DT^{S \times E}_{(v,w)} = \int_{[ M_{(v,w)}(S \times E)/E ]^{\text{vir}} } 1.
	\]
	
	\begin{thm} \label{thm:wallcrossing higher rank}
		Assume that $w \cdot v(y) = 0$.
		%that $\div(v \wedge w) = \div(w)$.
		For any fixed $F \in M(v)$ we have that
		\[
		\DT^{S \times E}_{(v,w)}
		=
		\GW^{M(v)}_{E, w'}
		-
		\chi(S^{[n]})
		\sum_{r | w} \frac{1}{r} (-1)^{w \cdot v}
		e^{\textup{vir}}(\Quot(F,u_r))
		\]
		where
		\begin{itemize}[itemsep=-2pt]
			\item $w' = - \langle w, - \rangle : v^{\perp} \to \BZ$ is the homology class induced by $w$,
			\item $u_r= -w/r - s_r v$ for the unique integer $s_r \in \BZ$ such that $0 \leq \rk(u_r) \leq \rk(v)-1$.
		\end{itemize} 
	\end{thm}

Because $v \cdot v(y) = 1$, the condition $w \cdot v(y) = 0$
%	The divisibility $\div(w)$ is the largest positive integer which divides $w$ in $\Lambda$. The condition $\div(v \wedge w) = \div(w)$ 
can always be achieved by replacing $w$ by $w+\ell v$ for some $\ell \in \BZ$.
	Since the stability condition on $S \times E$ is invariant under tensoring by line bundles pulled back from $E$,
	 $\DT^{S \times E}_{(v,w)}$ is invariant under this replacement.
	
	\subsection{Open questions}
	Let $S$ be a smooth projective surface with $H^0(S, \CO(-K_S)) \neq 0$.
	Let $F \in \Coh(S)$ be a Gieseker stable sheaf (with respect to some polarization).
	Then the Quot schemes $\Quot(F,u)$ 
	%where $u$ is of rank $0$ 
	carry a perfect obstruction theory, see Section~\ref{subsec:quot pot}.
	
	\begin{problem}
		Compute the virtual Euler number $e^{\vir}( \Quot(F,u) )$.
	\end{problem}
	
	Even for $F$ the ideal sheaf of a length $n$ subscheme this question
	needs (to the best of the authors knowledge) further investigation.
	In case $n=0$, that is quotients of the structure sheaf, we refer to \cite{OP} for some results.
	More generally, we can ask for the computation of the wall-crossing corrections in Nesterov's wall-crossing formula \cite{N1}.
	
	For K3 surfaces in upcoming work \cite{K3xE higher rank} the 
	higher rank Donaldson-Thomas theory of $S \times \E$ is
	related to the rank $1$ theory by derived auto-equivalences and wall-crossing. It will show that
	$\DT^{S \times E}_{(v,w)}$
	as we defined it above only depends on the pairings $v \cdot v$, $v \cdot w$ $w \cdot w$ and the divisibility $\div(v \wedge w)$.
	Moreover, by deformation theory of hyperk\"ahler varieties
	(the global Torelli theorem)
	the counts $\GW^{M(v)}_{E,w'}$ also only depends on the same data.
	Hence by Theorem~\ref{thm:wallcrossing higher rank}
	one finds the following:
	
	\begin{thm}(Dependent on \cite{K3xE higher rank})
		\label{thm:higher rank quot}
		Let $F$ be a Gieseker stable sheaf of positive rank and primitive Mukai vector $v = v(F)$ on a K3 surface $S$.
		Assume there exists a class $y \in K_{\mathrm{alg}}(S)$ such that $v \cdot y = 1$,
		%Let $M(v)$ be a proper moduli space of stable sheaves on a K3 surface $S$ of positive rank for which there exists a class $y \in K_{\mathrm{alg}}(S)$ such that $y \cdot v = 1$.
		%Let $F \in M(v)$ be a stable sheaf and
		and that $\Quot(F,u)$ is non-empty.
		Then the virtual Euler characteristic $e^{\text{vir}}(\Quot(F,u))$
		only depends on the following pairings in the Mukai lattice:
		\[ v \cdot v, \quad u \cdot v, \quad u \cdot u. \]
	\end{thm}
	
	Together with Theorem~\ref{thm:Correction term} this determines the Euler numbers $e^{\text{vir}}( \Quot(F,u))$ (since any such pair $(v,w)$ is isometric to a pair $((1,0,1-n), (0,\beta,m))$.
	However, it would be useful to have a more direct way to prove Theorem~\ref{thm:higher rank quot},
	since this would give another way to relate higher rank DT theory of $S \times \E$ to rank $1$.
	The natural pathway to proving the theorem is to apply an auto-equivalence
	which identifies $\Quot(F,u)$ with a Quot scheme of a rank $1$ object in the derived category,
	where the quotients are taken in the heart of some Bridgeland stability conditions.
	The theorem would then boil down to showing that
	the virtual Euler number of the Quot scheme stays invariant under change of hearts (the invariance under changing $F$ is provided already by deformation equivalence).
	
	The paper \cite{K3xE} proposed 8 different conjectures related to counting in K3 geometries. They were labeled
	\begin{center}
		A, B, C1, C2, D, E, F, G. \end{center}
	This paper here in combination with \cite{N1,N2, HAE} tackles Conjectures A and B of \cite{K3xE}. (Strictly speaking, we obtain Conjecture B only for DT invariants.)
	The same results also immediately imply Conjecture G.
	Conjecture C1 was proven by T.~Buelles \cite{Buelles}.
	The most difficult of the conjectures appear to be Conjectures C2 (multiple cover formula for Gromov-Witten theory of K3 surfaces, divisibility 2 solved in \cite{BB}),
	and Conjecture D (GW/DT correspondence for imprimitive classes).
	The remaining conjectures concern the matrix of quantum multiplication with a divisor on $S^{[n]}$. Conjecture E here was partially resolved in the work \cite{vIOP},
	which provided an explicit candidate for the matrix.
	This makes Conjecture F now the most accessible candidate on the list.
	%Conjecture F seems to be 
	
	\subsection{Plan of the paper}
	In Section~\ref{section:degree} we discuss our conventions for degree and prove a few basic lemmas about it.
	In Section~\ref{sec:Quot scheme integrals} 
	we explain a basic universality result for descendent integrals over nested Hilbert schemes (based on work of Gholampour and Thomas).
	We then express the virtual numbers of the Quot schemes $\Quot(F,u)$ in rank $1$
	as tautological integrals over the Hilbert scheme and use a degeneration argument to show that their generating series is of a certain form.
	Section~\ref{section:cap} concern the Donaldson-Thomas theory of the cap $S \times \p^1 / S_{\infty}$ and is the heart of the paper.
	We analyze the obstruction theory on the extremal component of the fixed locus by proving both vanishing of the contribution of most components and relate the remaining terms to the Quot integrals.
	The multiple cover formulas are taken up in Section~\ref{sec:multiple cover}. % by degeneration arguments.
	%where after introducing the right notation the notational hazard becomes manageable. % here.
	In Section~\ref{sec:Proofs} we put everything together and prove the theorems announced above.

	\subsection{Acknowledgements}
	I would like to thank Thorsten Beckmann, Denis Nesterov, Rahul Pandharipande and Richard Thomas for discussions related to this project.
	The link between the Hilbert scheme $S^{[n]}$ and Donaldson-Thomas theory
	(e.g. in the form of Theorem~\ref{thm:Wall  cross rank 1}) 
	is only possible due to the work \cite{N1,N2} of Denis Nesterov.
	The proofs of the multiple cover formulas rely on the beautiful second-cosection argument which was found by Rahul Pandharipande and Richard Thomas in \cite{PT}.
	I also thank the referees for useful comments which improved and corrected the presentation, in particular in Section~\ref{subsec:quasimap degree}.
	
The author was funded by the Deutsche Forschungsgemeinschaft (DFG) - OB 512/1-1,
and the starting grant 'Correspondences in enumerative geometry: Hilbert schemes, K3 surfaces and modular forms', No 101041491
 of the European Research Council.

	\section{Definitions of degree} \label{section:degree}
	\subsection{Mukai vector on K3 surfaces}
	Let $S$ be a K3 surface and consider the lattice $\Lambda = H^{\ast}(S,\BZ)$ endowed with the Mukai pairing
	\[ (x \cdot y) := - \int_S x^{\vee} y, \]
	where, if we decompose an element $x \in \Lambda$ according to degree as $(\rho,D,n)$, we have written $x^{\vee} = (\rho,-D,n)$.
	We will also write
	\[ \rk(x) = \rho, \quad c_1(x) = D, \quad v_2(x) = n. \]
	Given a sheaf or complex $E$ on $S$ the Mukai vector of $E$ is defined by
	\[ v(E) = \sqrt{\td_S} \cdot \ch(E) \in \Lambda. \]
	The relationship to the Euler characteristic is $\chi(E,F) = - v(E) \cdot v(F)$.
	
	\subsection{Mukai vector on $S \times C$}
	Let $C$ be a smooth curve. We naturally decompose the even cohomology
	\[ H^{2\ast}(S \times C, \BZ) = H^{\ast}(S,\BZ) 1_C \oplus H^{\ast}(S,\BZ) \omega \]
	where $1_C, \omega \in H^{\ast}(C)$ is the unit and the class of a point respectively.
	We denote the Mukai vector of a sheaf $F$ on $S \times C$ by
	\[ \ch(F) \sqrt{\td_S} = v(F) + w(F) \omega = (v(F),w(F)). \]
	
\subsection{Quasimap degree} \label{subsec:quasimap degree}
%Let $\Coh(v)$ be the stack of coherent sheaves on $S$ with Mukai vector $v$ and let
%\[ \Coh_{\mathfrak{r}}(v) = \Coh(v) \thickslash \BG_m \]
%be its rigidification.
Let $M(v)$ be a proper moduli space of Gieseker-stable sheaves in Mukai vector $v$ and let $M(v) \subset \Coh_{\mathfrak{r}}(v)$ be the rigidified stack of coherent sheaves of Mukai vector $v$ in which it $M(v)$ is embedded as an open substack. We assume that we are given an algebraic class $y \in K_{\mathrm{alg}}(S)$ with $v \cdot v(y) = 1$.
The class $y$ defines canonically a universal family $\BG$ over $\Coh_{\mathfrak{r}}(v)$,
which has the following property\footnote{Let $\Coh(v)$ be the stack of coherent sheaves on $S$ and consider the $\BG_m$-gerbe $\Coh(v) \to \Coh_{\mathfrak{r}}(v)$.
Let $\BF$ be the universal sheaf on $\Coh(v) \times S$ (which always exists and is canonical) and for $u \in K_{\mathrm{alg}}(S)$ let $\lambda_{\BF}(u) = \det \pi_{\ast}( \BF \otimes \pi_S^{\ast}(u) ) \in \Pic(\Coh(S))$. Then $\BF \otimes \lambda_{\BF}(-y^{\vee})^{-1}$ has $\BG_m$-weight zero and hence descends as the universal sheaf $\BG$ to $\Coh_{\mathfrak{r}}(v) \times S$.
% along the $\BG_m$-gerbe $\Coh(v) \to \Coh_{\mathfrak{r}}(v)$.

There is a subtle point: The construction of the line bundles $\lambda_{\BF}(u)$ and the universal sheaves $\BG$ require a resolution of $\pi_{\ast}(\BF \otimes \pi_S^{\ast}(u))$ which exists a priori only over finite type subschemes. Hence $\BG$ can be defined only over finite type open substacks of $\Coh_{\mathfrak{r}}(v)$, therefore globally only in an inductive way. Since the quasimaps we consider have fixed degree, they can be shown to map to a sufficiently large finite type substack of $\Coh_{\mathfrak{r}}(v)$ \cite{N1}. Hence for our purposes we may assume that $\BG$ and $\lambda(u)$ are globally defined. We refer to \cite{N1} for a discussion on this point and the exact conventions that we follow.}: If we define the morphism
\[ \lambda : K_{\mathrm{alg}}(S) \to \Pic(\Coh_{\mathfrak{r}}(v)), \quad \lambda(u) = \det \pi_{\ast}( \BG \otimes \pi_S^{\ast}(u) ) \]
where $\pi, \pi_S$ are the projections of $\Coh_{\mathfrak{r}}(v) \times S$ onto the factors, then
\begin{equation} \lambda(-y^{\vee}) = \CO. \label{3sdf} \end{equation}

From a hyperk\"ahler point of view the most natural way to construct cohomology classes on the stack $\Coh_{\mathfrak{r}}(v)$ is given by the Mukai morphism
\[ \theta : v^{\perp} \to H^2( \Coh_{\mathfrak{r}}(v)), \quad x \mapsto \left[ \pi_{\ast}\left( \ch(\BG) \sqrt{\td_S} \cdot x^{\vee} \right) \right]_{1}, \]
	where $[ - ]_{k}$ stands for taking the complex degree $k$ component of a cohomology class,
	i.e. the component in $H^{2k}$.
	By restricting the image to $M(v)$ we obtain an isomorphism of lattices:
	\[ \theta : v^{\perp} \xrightarrow{\cong} H^2(M(v), \BZ), \]
	where the right hand side carries the Beauville-Bogomolov-Fujiki form, see \cite[Sec.6.2]{HL}.
	
	We define the degree of a map $f : C \to \Coh_{\mathfrak{r}}(v)$ to be the morphism
	\[ \deg(f) : v^{\perp} \to \BZ \]
	given by $\deg(f)(y) = \int_C f^{\ast}( \theta(y) )$.
	\begin{lemma} \label{lemma:quasi degree}
		Let $f : C \to \Coh_{\mathfrak{r}}(v)$ be a quasimap (see \cite[Sec.3.1]{N1})
		and let $F$ be the associated sheaf on $S \times C$
		determined by the pullback of the universal family $\BG$.

Then $v(F) = v$ and the class $w(F)$ is determined by
\begin{equation} \label{deg eqn} \deg(f) = - \langle w(F), - \rangle \ \in \Hom( v^{\perp}, \BZ), \end{equation}
where $\langle x, - \rangle$ is the operator of pairing with $x$ in the Mukai lattice, and
\[ w(F) \cdot v(y) = 0. \]
	\end{lemma}
	\begin{proof}
By restricting $F$ to a fiber over $C$, we find $v(F) = v$.
Let $x \in v^{\perp}$. Then
		\begin{align*}
			\deg(f)(x) & = \int_{S \times C} \ch(F) \pi_S^{\ast}( x^{\vee} ) \sqrt{\td_S} \\
			&  = \int_{S \times C} (v(F) + w(F) \omega) \pi_S^{\ast}( x^{\vee} ) \\
			& = \int_S w(F) x^{\vee} \\
			& = - (w(F) \cdot x).
		\end{align*}
This shows \eqref{deg eqn}.
Similarly, by \eqref{3sdf} we have
\[ 0 = \int_{C} c_1( f^{\ast} \lambda(-y^{\vee}) )
=
\int_{S \times C} \ch( F \otimes \pi_S^{\ast}( -y^{\vee} )) \td_S 
=
\int_S w(F) \cdot v(-y^{\vee}) = w(F) \cdot v(y). \]
Since $v(y) \cdot v=1$, we see that $w(F)$ is determined by
$\deg(f)$ and $w(F) \cdot v(y)=0$.
	\end{proof}

\begin{rmk}
The divisibility $\div(\alpha)$ of a vector $\alpha \in (v^{\perp})^{\vee}$
is the largest positive integer $k$ such that $\alpha/k \in (v^{\perp})^{\vee}$.
Since $H^{\ast}(S,\BZ) \cong v^{\perp} \oplus \BZ v(y)$ we have
$\div(\deg(f)) = \div(w(F))$.
%One can also read off the divisibility of $\deg(f)$ from the pair $(v,w)$ by the formula
%\[ \div( \deg(f) ) = \div( v \wedge w) = \max\{  \div(w - kv) \,|\, k \in \BZ  \}. \]
%Since we do not need this later on, we omit the proof.
\end{rmk}

	\section{Quot scheme integrals} \label{sec:Quot scheme integrals}
	\subsection{Perfect obstruction theory} 
	\label{subsec:quot pot}
	Let $S$ be a smooth projective surface and let $F \in \Coh(S)$
	be a coherent sheaf which is of positive rank and Gieseker stable with respect to some ample class.
	Consider the Quot scheme $\Quot(F)$ and 
	let $\pi, \pi_S$ be the projections of $\Quot(F) \times S$ to the factors.
	We denote the universal quotient sequence on $\Quot(F) \times S$ by
	\[ 0 \to \CK \to \pi_S^{\ast}(F) \to \CQ \to 0. \]
	For sheaves (or complexes, or $K$-theory classes) $\CF_1, \CF_2$ on $\Quot(F) \times S$ we write
	\[ R\Hom_S(\CF_1, \CF_2) = R \pi_{\ast} R\hom( \CF_1, \CF_2 ). \]
	
	\begin{lemma} \label{lemma:canonical pot}
		Assume that $H^0(\CO(-K_S)) \neq 0$.
		Then the Quot scheme $\Quot(F)$ admits a (canonical) perfect obstruction theory
		with virtual tangent bundle $T^{\std} = R \Hom_S( \CK, \CQ )$.
	\end{lemma}
	\begin{proof}(Sketch)
		Consider a short exact sequence $0 \to K \to F \to Q \to 0$ defining a point in $\Quot(F)$, where $Q \neq F$,
		and apply $\Hom( -, Q)$. This gives
		\begin{equation} \Ext^1(K,Q) \to \Ext^2(Q,Q) \to \Ext^2(F,Q) \to \Ext^2(K, Q) \to 0. \label{xxy} \end{equation}
		Given a morphism $s : Q \to F$ we have a surjection
		$F \twoheadrightarrow Q \twoheadrightarrow \mathrm{Im}(s) \subset F$,
		which by stability of $F$ shows that $\mathrm{Im}(s) = 0$, so $s=0$.
		%By stability the sloper of $\mathrm{Im}(s)$ must be strictly
		%Since we have a surjection $F \to Q$ and $F$ is stable, the smallest slope appearing in the Harder-Narasimhan filtration of $Q$ must be strictly larger than the slope of $F$.
		Hence $\Hom(Q,F) = 0$. By our assumption on $S$ there exists an effective divisor $D \in |-K_S|$.
		Applying $\Hom(- , F)$ to $0 \to Q \to Q(D) \to Q|_{D} \to 0$ implies then that $\Hom(Q(D), F) = 0$.
		We conclude that $\Ext^2(F,Q) = \Hom(Q, F \otimes K_S)^{\vee} = \Hom(Q(D), F)^{\vee}$ vanishes.
		The existence then follows by standard methods, e.g. \cite{MO, OP}.
	\end{proof}
	\begin{rmk}
		The same argument works for any surface $S$ if $F$ is the ideal sheaf of a zero-dimensional subscheme.
	\end{rmk}

	\subsection{K3 surfaces} \label{subsec:quot k3}
	From now on let $S$ be a K3 surface. Let $\Quot(F,u)$ be the Quot scheme
	parametrizing quotients $F \to Q$ with Mukai vector $v(Q) = u$.
	We always assume that $u$ is chosen such that the Quot scheme is non-empty
	and that $u \notin \{ v(F), 0 \}$.
	
	\begin{lemma}
		The canonical perfect obstruction theory on $\Quot(F,u)$ admits a surjective cosection $h^1(T^{\std}) \to \CO$.
		
	\end{lemma}
	\begin{proof}
		We compose the first map in \eqref{xxy} with the trace map
		\[ H^1(R\Hom(K,Q)) = \Ext^1(K,Q) \to \Ext^2(Q,Q) \xrightarrow{tr} H^2(S, \CO_S) = \CO. \]
		Since the trace map is Serre dual to the inclusion $H^0(S, \CO_S) \hookrightarrow \Hom(Q,Q)$,
		it is surjective.
	\end{proof}
	
	Hence the standard virtual class $[\Quot(F,u)]^{\std}$ which is of dimension
	$\chi(K,Q) = u \cdot (u-v)$ vanishes.
	Using co-section localization by Kiem-Li \cite{KL} we obtain a reduced virtual cycle:
	\[ [ \Quot(F,u) ]^{\vir} \in A_{\vd}(\Quot(F,u)). \]
	It is associated to the (reduced) virtual tangent bundle
	$T^{\vir}_{\Quot(F,u)} =  R \Hom_S( \CK, \CQ ) + \CO$,
	and hence of dimension $\vd = \rk(T^{\vir}_{\Quot(F,u)}) = u \cdot (u-v) + 1$.
	We write
	\[ e^{\vir}( \Quot(F,u) ) =
	\int_{ [ \Quot(F,u) ]^{\vir} } c_{\vd}( T^{\vir}_{\Quot(F,u)} ) \]
	for the virtual Euler characteristic of the moduli space in the sense of Fantechi--G\"ottsche \cite{FG}.

	\subsection{Nested Hilbert schemes}
	Let $\beta \in \mathrm{NS}(S)$ be an effective curve class
	and let $n_1, n_2 \geq 0$ be integers.
	Consider the nested Hilbert scheme, where we follow the notation of \cite{GT2},
	\[ S_{\beta}^{[n_1, n_2]}
	=
	\{ I_1(-D) \subseteq I_2 \subset \CO_S : [D] = \beta,
	\mathrm{length}( \CO_S / I_i) = n_i. \}. \]
	There exists a natural embedding
	\begin{equation} j : S_{\beta}^{[n_1, n_2]} \hookrightarrow S^{[n_1]} \times S^{[n_2]} \times \p \label{inclusion} \end{equation}
	where the linear system in class $\beta$ is denoted by
	\[ \p = \p(H^0(\CO(\beta))) \]
	Let $\CI_i \subset \CO$
	be the ideal sheaf of the universal subscheme $\CZ_i \subset S^{[n_i]} \times S$.

	\begin{thm}[{\cite{GT2}}] \label{thm:nested hilb}
		There exists a (reduced) perfect obstruction theory on $S_{\beta}^{[n_1, n_2]}$
		with virtual tangent bundle
		\[ 
		T_{S_{\beta}^{[n_1, n_2]}}^{\vir}
		=
		- R \Hom_{S}(\CI_1, \CI_1)_0 - R \Hom_S( \CI_2, \CI_2 )_0
		+ R \Hom_{S}( \CI_1, \CI_2 \otimes \CO(\beta)) \otimes \CO_{\p}(1) - \CO \]
		where $R \Hom_S( \CI_i, \CI_i)_0 = \Cone( R \Gamma(S, \CO_S) \to R \Hom_S( \CI_i, \CI_i ))$.
		
		The associated virtual cycle
		$[ S_{\beta}^{[n_1, n_2]} ]^{\text{vir}}$
		is of dimension $n_1 + n_2 + \beta^2/2 + 1$ and satisfies
		\begin{multline*}
			j_{\ast} [ S_{\beta}^{[n_1, n_2]} ]^{\text{vir}} \\
			=
			c_{n_1 + n_2}\Big( R\Gamma(\CO(\beta)) \otimes \CO_{\p}(1) - R \Hom_S( \CI_1, \CI_2 \otimes \CO(\beta) ) \otimes \CO_{\p}(1) \Big) \cdot c_1(\CO_{\p^1}(1))^{h^1(\CO(\beta))}
		\end{multline*}
	\end{thm}
	\begin{proof}
		The first claim follows directly from Theorem 4.16 in \cite{GT2},
		which also shows the second claim whenever $H^1(\CO(\beta)) = 0$ (we can take $A=0$). In the general case, where $H^1(\CO(\beta))$ may be non-zero\footnote{A basic example is an elliptic K3 surface $S \to \p^1$ and $\beta= m f$ where $f$ is the fiber class. The linear system $|\CO(\beta)|$ is of dimension $m$, while $\chi(\CO(\beta)) = 1$.}, 
		we apply Corollary 4.22 of \cite{GT2} which gives:
		\begin{equation} \label{3sdfis111}
			j_{\ast} [ S_{\beta}^{[n_1, n_2]} ]^{\text{vir}}
			=
			c_{n_1 + n_2}( R\Gamma(\CO(\beta)) \otimes \CO_{\p}(1) - R \Hom_S( \CI_1, \CI_2 \otimes \CO(\beta) ) \otimes \CO_{\p}(1) ) \cdot
			[ S_{\beta} ]^{\vir}
		\end{equation}
		where for a class $\gamma \in \mathrm{NS}(S)$ we write
		$S_{\gamma} = \p( H^0(\CO(\gamma)))$ for the Hilbert scheme of curves in class $\gamma$ (which of course for a K3 surface is just the linear system).
		The virtual class $[ S_{\beta} ]^{\vir}$ in \eqref{3sdfis111} is the natural one appearing in Seiberg-Witten theory and can be identified with the first degeneracy locus of the complex $R \Gamma(\CO(\beta))$.
		Concretely, it is described as follows (see also \cite[Sec.2]{KT2}).
		Choose a fixed ample divisor $A \subset S$ such that $H^{\geq 1}(\CO(\beta+A)) = 0$. Let $\gamma = \beta + [A] \in \mathrm{NS}(S)$.
		There exists an embedding
		\[ S_{\beta} \to S_{\gamma}, \quad C \mapsto C + A. \]
		Its image consists of those divisors $D = \{ s = 0 \}$ which contain $A$,
		or equivalently, for which the composition $\CO_S \xrightarrow{s} \CO_S(\gamma) \to \CO_S(\gamma)|_{A}$ vanishes.
		Globally, let $\CD \subset S_{\gamma} \times S$ be the universal divisor and let $\pi : S_{\gamma} \times S \to S_{\gamma}$ be the projection.
		There exists a universal section $s : \CO \to \CO(\CD)$ which yields the sequence
		$\CO \to \CO(\CD) \to \CO(\CD)|_{S_{\gamma} \times A}$.
		Pushing forward, and using that $\CO(\CD) = \pi_S^{\ast}(\CO(\gamma)) \otimes \CO_{S_{\gamma}}(1)$ we find that $S_{\beta}$ is naturally cut out by a section of
		\[ 
		\pi_{\ast}(\CO(\CD)|_{S_{\gamma} \times A}) = H^0(\CO(\beta+A)|A) \otimes \CO(1).
		\] 
		The associated virtual class $[S_{\beta}]^{\vir}$ is the localized Euler class.
		Using the sequence 
		\[ 0 \to H^0(\CO(\beta)) \xrightarrow{f} H^0(\CO(\beta+A)) \to H^0(\CO(\beta+A)|_{A})
		\to H^1(\CO(\beta)) \to 0 \]
		and that $S_{\beta}$ is cut out by $\mathrm{Cokernel(f)} \otimes \CO(1)$ one obtains that
		\[ S_{\beta}^{\vir} = e(H^1(\CO(\beta)) \otimes \CO(1)) \cap [ S_{\beta} ]. \]
		This shows the claim.
	\end{proof}
	
	Let $\CZ \subset S^{[n]} \times S$ be the universal subscheme of the Hilbert scheme of points $S^{[n]}$.
	Given $\alpha \in H^{\ast}(S)$ we define the descendants
	\[ \tau_k^{S^{[n]}}(\alpha) := \pi_{\ast}( \ch_{2+k}(\CO_{\CZ}) \pi_S^{\ast}(\alpha) ) \]
	where $\pi, \pi_S$ are the projections of $S^{[n]} \times S$ to the factors.
	
	We consider integrals over $S_{\beta}^{[n_1, n_2]}$ of polynomials in the pullback of the classes
	%\quad k_i \geq 0, \alpha_i \in H^{\ast}(S), \]
	\[ \tau_{k_1}^{S^{[n_1]}}( \alpha_1 ),
	\quad \tau_{k_2}^{S^{[n_2]}}( \alpha_2 ),
	\quad z = c_1(\CO_{\p}(1)), \quad k_i \geq 0, \alpha_i \in H^{\ast}(S), \]
	where the pullback is by the composition
	of the inclusion \eqref{inclusion} with the projection to the factors.
	By Theorem~\ref{thm:nested hilb} we can reduce any such integral to an integral of such type of classes  on $S^{[n_1]} \times S^{[n_2]} \times \p$.
	Integrating out $\p$ and using \cite{EGL} one obtains:
	%(see \cite{GNY} for the product case)
	%one obtains the following:
	
	%\[ \tau_{k_{i, 1}}^{S^{[n_1]}}( \alpha_{i,1} ),
	%\quad \tau_{k_{j,2}}^{S^{[n_2]}}( \alpha_{j,2} ),
	%\quad z = c_1(\CO_{\p}(1)) \]
	%for some $k_{i,1}, k_{j,1} \geq 0$ and $\alpha_{i,1}, \alpha_{j,2} \in H^{\ast}(S)$.
	
	\begin{thm}(Universality) \label{thm:universality}
		Let $P$ be a polynomial.
		Let $\alpha_{i,1}, \alpha_{j_2} \in H^{\ast}(S)$ be homogeneous classes and let $k_{i,1}, k_{j,2} \geq 0$ be integers. 
		Then the integral
		\[ 
		\int_{ \left[ S_{\beta}^{[n_1, n_2]} \right]^{\text{vir}} }
		P( \tau_{k_{i,1}}^{S^{[n_1]}}( \alpha_{i,1} ),\ 
		\tau_{k_{j,2}}^{S^{[n_2]}}( \alpha_{j,2} ),
		\ z)
		\]
		%for some polynomial $P$, 
		depends upon $(S,\beta, \alpha_{i,1}, \alpha_{j,2})$
		only through the intersection pairings of the classes
		$\beta, \alpha_{i_1}, \alpha_{j_2}, 1, \pt$.
	\end{thm}

	\subsection{Universality}
	Let $\eta \in S^{[n]}$ be a fixed length $n$ subscheme and let $I_{\eta} \subset \CO_S$ be its ideal sheaf.
	We consider the Quot scheme 
	$\Quot(I_{\eta}, u)$ 
	for Mukai vector $u = (0,\beta,m)$
	with $\beta \in \mathrm{NS}(S)$ effective.
	
	The Quot scheme $\Quot(I_{\eta}, u)$ parametrizes sequences of the form
	\[ 0 \to I_{z}(-\beta) \to I_{\eta} \to Q \to 0 \]
	for some $z \in S^{[n_1]}$, where one computes that
	\[ n_1 = m + n + \beta^2/2. \]
	Hence the Quot scheme is naturally a subscheme of the nested Hilbert scheme:
	\begin{equation} \Quot(I_{\eta}, (0,\beta,m)) = \pi_2^{-1}(\eta) \subset S_{\beta}^{[n_1, n_2]} \label{identification}
	\end{equation}
	where $\pi_2 : S_{\beta}^{[n_1, n_2]} \to S^{[n_2]}$
	is the projection and $n_2=n$.
	
	\begin{lemma}
		We have the following comparison
		of virtual cycles:
		\[
		[ \Quot(I_{\eta},u) ]^{\vir}
		=
		\iota_{\eta}^{!} [ S_{\beta}^{[n_1, n_2]} ]^{\text{vir}}.
		\]
		where $\iota_{\eta} : \{ \eta \} \to S^{[n_2]}$ is the inclusion.
	\end{lemma}
	\begin{proof}
		Let $\CI$ denote the universal ideal sheaf on $S^{[n]} \times S$
		and let $\pi : S^{[n]} \times S \to S^{[n]}$ be the projection.
		We can identify the $\pi$-relative Quot scheme $\Quot( \CI / S^{[n]}, u)$
		(whose fiber over a point $\eta \in S^{[n]}$ is $\Quot(I_{\eta}, u)$)
		with the nested Hilbert scheme:
		\begin{equation} \Quot( \CI/S^{[n]}, u ) = S_{\beta}^{[n_1, n_2]}. 
			\label{quot=hilb} \end{equation}
		The left hand side carries the natural perfect obstruction theory
		of Lemma~\ref{lemma:canonical pot} taken relative to the base $S^{[n]}$.
		Its virtual tangent bundle is
		\[
		T_{\Quot( \CI/S^{[n]}, u )}^{\vir} = R \Hom_S( \CK, \CQ ) + \CO
		+ \Ext^1_S( \CI, \CI )
		\]
		and its virtual class satisfies
		\[
		\iota_{\eta}^{!} [ \Quot( \CI/S^{[n]}, u ) ]^{\vir} =
		[ \Quot( \CI_{\eta}, u ) ]^{\vir}.
		\]
		
		Under the identification \eqref{quot=hilb} we have
		\begin{equation} \label{KQ identification} \CK = \CI_1(-\beta) \otimes \CO_{\p}(-1), \quad
			\CQ = \CI_2 - \CK. \end{equation}
		By the first claim in Theorem~\ref{thm:nested hilb} a small calculation shows that
		\[ T_{\Quot( \CI/S^{[n]}, u )}^{\vir} = T_{S_{\beta}^{[n_1, n_2]}}^{\vir}. \]
		Since the virtual class depends on the perfect obstruction theory only through the $K$-theory class of the virtual tangent bundle we get
		\[ [ \Quot( \CI/S^{[n]}, u ) ]^{\vir} = [S_{\beta}^{[n_1, n_2]} ]^{\vir}. \]
	\end{proof}
	
	By Theorem~\ref{thm:nested hilb}, the above lemma and \eqref{KQ identification} we obtain:
	\begin{align}
		e^{\vir}( \Quot(I_{\eta}, u) )
		& = \int_{[ \Quot(F,u) ]^{\vir}} c_{\beta^2+m+1}\left( R \Hom_S( \CK, \CQ ) + \CO \right) \notag \\
		& = \int_{[ S_{\beta}^{[n_1, n_2]} ]^{\vir}} c_{\beta^2+m+1}\left( R \Hom_S( \CK, \CQ ) + \CO \right) \pi_2^{\ast}[ F ]. \label{xxggrter}
	\end{align}
	Applying Grothendieck-Riemann-Roch to rewrite the term $e\left( R \Hom_S( \CK, \CQ ) + \CO \right)$ in descendants,
	and applying Theorem~\ref{thm:universality} we conclude the following:
	
	\begin{prop} \label{prop:div independence for Q}
		The integrals $\mathsf{Q}_{n,(\beta,m)} = e^{\vir}( \Quot(I_{\eta}, u) )$
		only depends on $\beta$ via the square $\beta \cdot \beta = 2h-2$.
		We write $ \mathsf{Q}_{n,h,m} :=  \mathsf{Q}_{n,\beta,m}$ from now on.
	\end{prop}
	
	\begin{rmk}
		One can be more explicit in what kind of tautological integral over $S^{[n_1]}$ one obtains when computing
		$\mathsf{Q}_{n, (\beta,m)}$. By similar arguments as in \cite{OP} one proves:
		%following \cite{OP}) that
		%$S^{[n_1]}$ is holomorphic-symplectic, so
		%$T_{S^{[n_1]}} \cong \Omega_{S^{[n_1]}}$,
		%and even-dimensional. % (see \cite{OP} where we learned this trick).
		%In the last term we treated $z$ as a formal variable.
		\[
		\begin{aligned}
			e^{\text{red}}( \Quot(I_{\eta}, u) )
			%\mathsf{Q}_{h,n_1, n_2}
			%=
			%e^{\text{red}}( \Quot(F,(0,\beta,m)) )
			%& =
			%\int_{S^{[n_1]} \times \p}
			%c( T_{S^{[n_1]}})
			%\frac{ c_{n_1+n_2}( (x^{[n_1]})^{\vee}(1) + n_2 \CO(1) ) }{ c((x^{[n_1]})^{\vee}(1) + n_2 \CO(1)) }
			%(1+z)^{h+1} \\
			%& =
			%\mathrm{Coeff}_{t^{n_1+n_2}}
			%\int_{S^{[n_1]} \times \p}
			%c( T_{S^{[n_1]}})
			%\frac{ c_t( x^{[n_1]} \otimes e^z ) }{ c( x^{[n_1]} \otimes e^z) }
			%(1+z)^{h+1-n_2} (1+t z)^{n_2} \\
			& =
			\mathrm{Coeff}_{t^{n_1+n_2} z^{-1}}
			%\mathrm{Res}_{z=0}
			\int_{S^{[n_1]}}
			c( T_{S^{[n_1]}})
			\frac{ c_t( x^{[n_1]} \otimes e^z ) }{ c( x^{[n_1]} \otimes e^z) }
			\left( \frac{1+z}{z} \right)^{h+1-n_2} (z^{-1}+t)^{n_2}
		\end{aligned}
		\]
		where $z, t$ are formal variables,
		$x = \CO_S(-\beta) - \CO_{\eta}$,
		and $x^{[n]} = \pi_{\ast}(\pi_S^{\ast}(x) \otimes \CO_{\CZ}) \in K_{\mathrm{alg}}(S^{[n]})$ is the tautological class associated to $x$.
		If $n_2 = 0$, the above discussion specializes to Section 5.2 in \cite{OP}.
		We do not need this expression later on.
	\end{rmk}

	\subsection{Multiplicativity}
	Our goal here is to prove the following structural statement for $\mathsf{Q}_{n,h,m}$.
	Define the series
	\[
	\mathsf{Q}_n(p,q) = \sum_{h \geq 0} \sum_{m \in \BZ} \mathsf{Q}_{n,h,m} q^{h-1} p^m
	\]
	\begin{prop} \label{prop:Qmultiplicativity}
		There exists power series $F_1 \in \BQ((p, p^{-1}))[[q]]$ and $F_2 \in q^{-1} \BQ((p, p^{-1}))[[q]]$ 
		such that
		for any $n \geq 0$ we have
		\[ \mathsf{Q}_n(p,q) = F_1^n F_2. \]
	\end{prop}
	\begin{proof}
		Let $S \to \p^1$ be an elliptic K3 surface with section $B \subset S$ and take $\beta_h = B+hF$.
		Let $E \subset S$ be a fixed smooth fiber over the point $x \in \p^1$.
		We apply the Li-Wu degeneration formula \cite{LiWu} to the degeneration
		\begin{equation} S \rightsquigarrow S \cup_{E} (\p^1 \times E) \cup_E \ldots \cup_E (\p^1 \times E) \label{deg} \end{equation}
		where there are $n+1$ copies of $\p^1 \times E$ and the $i$-th copy is glued along the divisor\footnote{We write $E_z$ to denote the fiber over $z \in \p^1$ of the projection $E \times \p^1 \to \p^1$.} $E_{\infty}$ to the divisor $E_0$ in the $(i+1)$-th copy (for $i=1, \ldots, n$). Moreover, $S$ is glued along $E$ to $E_0$ in the first copy.
		This type of degeneration plays a crucial role in \cite{MPT}.
		
		Let $p \in \p^1 \times E$ be a point disjoint from the relative divisor $E_{0,\infty} = \{ 0, \infty \} \times E$,
		and let $v_{d,m} = (0, [\p^1] + d [E], m)$.
		Let $\Quot_{\p^1 \times E / E_{0, \infty}}( I_p, v_{d,m} )$ be the Quot scheme of the relative pair $(\p^1 \times E, E_{0, \infty})$.
		It parametrizes quotients $r^{\ast} I_p \twoheadrightarrow Q$ where $r_{\ast} \ch(Q) = v_{d,m}$
		and $r$ is the canonical projection from an expanded degeneration to $\p^1 \times E$ along $E_{0,\infty}$, see \cite{LiWu}.
		Since the support of $Q$ meets the relative divisor $E_0$ with multiplicity $1$ there exists evaluation morphisms
		$\ev_0 : \Quot_{\p^1 \times E / E_{0, \infty}}( I_p, v_{d,m} ) \to E_0$.
		Using the obstruction theory of Lemma \ref{lemma:canonical pot}, the fiber
		\[ \Quot_{\p^1 \times E / E_{0, \infty}}( I_p, v_{d,m} )_0 = \ev_0^{-1}(0_E) \]
		is seen to have virtual tangent bundle $R \Hom_{\p^1 \times E}( \CK, \CQ ) - \CO$.
		Define the generating series:
		\[ \mathsf{Q}_{\pt}(\p^1 \times E/E_{0,\infty}) = \sum_{d \geq 0} \sum_{m \in \BZ} q^d p^m e^{\vir}\left( \Quot_{\p^1 \times E / E_{0, \infty}}( I_p, v_{d,m} )_0 \right) \]
		
		Similarly, let $\Quot_{\p^1 \times E / E_{0}}( \CO_{\p^1 \times E}, v_{d,m} )_0$ be the relative Quot scheme on $\p^1 \times E / E_0$ which
		parametrizes quotients $\CO \to Q$ with $r_{\ast} \ch(Q) = v_{d,m}$ such that the restriction to the relative fiber 
		$(\CO_{\p^1 \times E} \to Q)|_{E_0}$ is isomorphic to $\CO_{E_0} \to \CO_{0}$.
		Define
		\[ \mathsf{Q}(\p^1 \times E/E_{0}) = \sum_{d \geq 0} \sum_{m \in \BZ} q^d p^m e^{\vir}\left( \Quot_{\p^1 \times E / E_{0}}( \CO_{\p^1 \times E}, v_{d,m} )_0 \right). \]
		Finally, let
		$\Quot_{S / E}( \CO_S, (0,\beta_h,m)$ parametrize quotients $\CO \twoheadrightarrow Q$ on the pair $(S,E)$ with $\ch(Q) = (0,\beta_h, m)$.
		Define
		\[
		\mathsf{Q}(S/E) = \sum_{h \geq 0} \sum_{m \in \BZ} q^{h-1} p^m e^{\vir}\left( \Quot_{S / E}( \CO_S, (0,\beta_h,m) ) \right).
		\]
		
		Let $I_{\eta}$ be the ideal sheaf of a length $n$ subscheme $\eta = x_1 + \ldots + x_n$ for distinct points $x_i \in S$.
		We can choose a simple degeneration $\pi : \CS \to C$ over a smooth curve $C$ 
		together with disjoint sections $p_i : C \to \CS$ such that
		\begin{enumerate}
			\item[(i)] Over the point $c_0 \in C$, $(\pi^{-1}(c_0), p_1(c_0), \ldots, p_n(c_0)) = (S, x_1, \ldots, x_n)$
			\item[(ii)] Over the point $c_1 \in C$, the fiber $\pi^{-1}(c_1)$ is the surface on the right of \eqref{deg}, 
			and $p_i(c_1)$ is a point on the $i$-th copy of $\p^1 \times E$ away from the relative divisors.
		\end{enumerate}
		We then apply the Li-Wu degeneration formula to
		$\Quot_{\CS \to C}( \BI, (0, \beta_h, m))$,
		the Quot scheme relative to the base $C$,
		where $\BI$ is the ideal sheaf of the union $\cup_i p_i(C)$.
		One finds that:
		\[
		\mathsf{Q}_n(p,q)
		=
		\mathsf{Q}(S/E) \cdot \left(\mathsf{Q}_{\pt}(\p^1 \times E/E_{0,\infty}) \right)^n \cdot  \mathsf{Q}(\p^1 \times E/E_{0}),
		\]
		which implies the claim.
		(The main point is that the integrand splits nicely: If $0 \to K \to I_{p_1, \ldots, p_n} \to Q \to 0$ 
		is a quotient sequence on the right side of \eqref{deg}, and $Q_i, K_i$ are the restriction to the $i$-th component (with $S$ the $0$-th component),
		then by applying $\Hom(K, - )$ to the sequence $0 \to Q \to \sum_{i=0}^{n+1} Q_i \to \sum_{j=0}^{n} Q|_{x_j} \to 0$ and using adjunction one finds that
		\[ R \Hom(K,Q) + \CO = ( R \Hom(K_0, Q_0) + \CO ) + \sum_{i=1}^{n+1} (R \Hom(K_i, Q_i) - \CO ). \ ) \]
	\end{proof}

	\section{Analysis of the cap geometry}
	\label{section:cap}
	\subsection{Overview}
	Let $S$ be a K3 surface and let $\Lambda = H^{\ast}(S,\BZ)$ be the Mukai lattice, see Section~\ref{section:degree}.
	Let $v \in \Lambda$ be a primitive vector of positive rank satisfying $v \cdot v = 2n-2$.
	Let $H$ be a polarization on $S$ and assume that the moduli space of $H$-Gieseker stable sheaves $M(v)$ is proper.
	We also fix a class $y \in K_{\mathrm{alg}}(S)$ such that $y \cdot v = 1$, which implies that $M(v)$ is fine and admits a canonical universal family (see Section~\ref{subsec:quasimap degree}).
	
	%and let $H$ be a polarization on $S$ such that the moduli space of $H$-Gieseker-stable sheaves $M(v)$ proper.
	%We assume that $M(v)$ is fine.
	%Assuming the choice of a $v$-generic polarization $H$ on $S$
	%(so that the moduli space $M(v)$ of $H$-stable sheaves is proper), let
	Consider the moduli space
	\[ M_{v,w}(S \times \p^1/S_{\infty}) \]
	parametrizing torsion-free, generically $H$-stable sheaves $E$ of fixed determinant\footnote{We refer to \cite[Section 3]{N1} for the precise definition of generically $H$-stable and how the determinant is fixed. Essentially, generically $H$-stable is the condition that the restriction of the sheaf to the generic fiber over $\p^1$ (or in case that the sheaf is defined over a degeneration of $\p^1$, the restrictions to all generic fibers of this degeneration) is $H$-stable.
		The sheaf $E$ has fixed determinant if $\det p_{\ast}( E \cdot \pi_S^{\ast}(y) ) \cong \CO$, where $p$ is the projection to the curve.}
	on the relative geometry $(S \times \p^1, S_{\infty})$
	with Mukai vector $v(E) = (v,w)$. By \cite{N1} the moduli space is proper,
	and because of the existence of the class $y$, it is fine and admits a canonical universal family.
	By Lemma~\ref{lemma:quasi degree} the class $w$ satisfies
\begin{equation} w \cdot v(y) = 0. \label{dersdf} \end{equation}
% after tensoring with $\CO_{\p^1}(k)$ for some $k$
%	we can and will always assume that
%	\[ \div(v \wedge w) = \div(w) \]
%	and that $v \wedge w \neq 0$.
	
	Let $\CE$ be the universal sheaf over $M_{v,w}(S \times \p^1/S_{\infty})$. % (which we assume exists, the assumption can be removed by working with twisted sheaves).
	The standard (non-reduced) perfect obstruction theory of $M_{v,w}(S \times \p^1/S_{\infty})$ has virtual tangent bundle
	\[ T^{\text{vir}} = R \Hom_{S \times \p^1}(\CE,\CE)_0[1] \]
	and is of virtual dimension $v^2 + \chi(S,\CO_S) = 2n$.
	The associated standard virtual class vanishes because of the existence of a cosection \cite{N2}.
	By work of Kiem-Li \cite{KL} there exists a reduced virtual class.
	The reduced virtual dimension is $2n+1$.
	
	The group $\BC^{\ast}$ acts on the base $\p^1$ with tangent weight $-t$ at the point $0 \in \p^1$, where we let $\Ft$ be the trivial line bundle with $\BC^{\ast}$-action of weight $1$ and set $t = c_1(\Ft)$.
	We obtain an induced action on the moduli space $M_{v,w}(S \times \p^1/S_{\infty})$.
	%together with equivariant obstruction theories. 
	Let
	\begin{equation} M_{\mathrm{ext}} \subset M_{v,w}(S \times \p^1/S_{\infty})^{\BC^{\ast}} \label{mext inclusion} \end{equation}
	be the component of the fixed locus
	which parametrizes sheaves on $S \times \p^1$ (i.e. sheaves on an expanded degeneration are excluded). We call $M_{\mathrm{ext}}$
	the \emph{extremal} component.
	
	The goal of this section is to analyze the contribution of this extremal component to the
	Donaldson-Thomas invariant.
	We first prove that only 'single-jump' loci contribute to the invariants,
	and then relate these contributions to Quot scheme integrals.
	This determines the wall-crossing term in the quasimap wall-crossing.

	\subsection{Virtual class}
	Let $E \in M_{\mathrm{ext}}$ and write
	\[ \p^1 = \BA^1_{\infty} \cup \BA^1 \]
	where $\BA^1_{\infty}, \BA^1$ is the standard affine chart around $\infty$ and $0$ respectively.
	Since $E$ is generically stable and its restriction to $S_{\infty}$ is stable,
	we have
	\[ E|_{S \times \BA^1_{\infty} } = \pi_S^{\ast}(F) \]
	for some stable sheaf $F \in M(v)$.
	Over $0$, we identify the equivariant sheaf $E|_{S \times \BA^1}$ with a graded $\CO_S[x]$-module on $S$,
	\begin{equation} \label{decompo}
		E|_{S \times \BA^1} = \bigoplus_{i \geq i_0} E_i \Ft^i
	\end{equation}
	for some $E_i \in \Coh(S)$ and $i_0 \in \BZ$. 
	%Since $E$ is finitely generated, there exists $i_0$ such that $E_i = 0$ for $i<i_0$
	Since $E|_{S \times \BA^1}$ is finitely generated,
	%By finite generation 
	there exists $r$ such that
	$E_{i_0+r} = E_{i_0 + r+1} = \ldots =: F$.
	Hence our sheaf takes the form
	\[
	E|_{S \times \BA^1} = E_{i_0} \Ft^{i_0} \oplus E_{i_0+1} \Ft^{i_0+1} \oplus \cdots \oplus E_{i_0 + r-1} \Ft^{i_0 + r-1} \oplus F \Ft^{i_0 + r} \oplus F \Ft^{i_0 + r+1} \oplus F \Ft^{i_0+r+2} \oplus \ldots  \,.
	\]
	%Hence by twisting $E$ by an appropriate power of $\CO_{\p^1}(-1)$ we can normalize $E$ such that $i_0 = 0$.
	%for some $r \geq i$ where $E_i \in \Coh(S)$ and $\Ft$ is the trivial line bundle on $S$ with $\BC^{\ast}$-action of weight $1$.
	By the assumption that $E$ is torsion-free, multiplication by $x$  yields the injective morphisms
	\[ f_i : E_i \hookrightarrow E_{i+1}. \]
	Thus associated to $E$ we have the flag of subsheaves
	\[ E_{\bullet} = \left( E_{i_0} \subseteq E_{i_0+1} \subseteq E_{i_0+2} \subseteq \ldots  \subseteq E_{i_0+r} = F \right). \]
	
	The stabilization parameter $r$ can be chosen uniformly on each connected component of $M_{\mathrm{ext}}$
	(because it only depends on the Chern classes $\ch(E_i)$).
	Since there are only finitely many connected components, we hence may choose an $r$ that is a stabilization parameter
	for all sheaves $E \in M_{\mathrm{ext}}$.
	
	\begin{prop} \label{prop:fixed virtual class}
		(a) The fixed part of the restriction of $T^{\text{vir}}$ to $M_{\mathrm{ext}}$ is given by
		%  There exists a perfect obstruction theory on $M_{\mathrm{ext}}$ with virtual tangent bundle
		\begin{equation} \left( T^{\text{vir}}|_{M_{\mathrm{ext}}} \right)^{\textup{fixed}} 
			\cong \mathrm{Cone}\left( \bigoplus_{i=i_0}^{i_0+r} R \Hom_S(E_i, E_i)_0 \xrightarrow{\delta} \bigoplus_{i=i_0}^{i_0+r-1} R \Hom_S(E_i, E_{i+1})_{0} \right) \label{fix pot} \end{equation}
		where
		\begin{align}
			R \Hom_S(E_i, E_i)_0 & = \Cone( R \Gamma(S, \CO_S) \xrightarrow{\id} R \Hom_S(E_i, E_i) ) \\
			R \Hom_S(E_i, E_{i+1})_0 & = \Cone( R \Gamma(S, \CO_S) \xrightarrow{\id} R \Hom_S(E_i, E_i) \xrightarrow{f_i \circ ( - )} R \Hom_S(E_i, E_{i+1}) ) \label{defn eiej0}
		\end{align}
		and $\delta$ is induced by the map that sends a tuple $(\alpha_i)_i \in \oplus_i R \Hom(E_i, E_i)$ to
		$(\alpha_{i+1} \circ f_i - f_i \circ \alpha_i)_i$ where $f_i : E_i \to E_{i+1}$ is the inclusion map.
		
		(b) The $K$-theory class of the moving part of the restriction of $T^{\text{vir}}$ to $M_{\mathrm{ext}}$ is
		\begin{align*} 
			\left( T^{\text{vir}}|_{M_{\mathrm{ext}}} \right)^{\textup{mov}} 
			& = 
			\sum_{i \geq i_0} \sum_{k \geq 1} 
			\left( \begin{array}{c}
				- \Ft^{-k} \otimes R\Hom_S( E_{i+k} - E_{i+k-1}, E_i) \\
				+ \Ft^{k} \otimes R\Hom_S( E_{i+k+1} - E_{i+k}, E_i)^{\vee}
			\end{array} \right).
		\end{align*}
		%
		%\noindent
		%An alternative formulation is
		%\[ \left( T^{\text{vir}}|_{M_{\mathrm{ext}}} \right)^{\textup{fixed}} 
		%= \mathrm{Cone}\left(  \left( \bigoplus_{i=0}^{r} R \Hom_S(E_i, E_i) \right)_0 \to \bigoplus_{i=0}^{r-1} R \Hom_S(E_i, E_{i+1}) \right). \]
		%\vspace{3pt} \noindent (b) The $K$-theory class of $T^{\text{vir}}_{M_{\mathrm{ext}}}$ is equal to the class of
		%the fixed part of $T^{\text{vir}}|_{M_{\mathrm{ext}}}$.
	\end{prop}
	\begin{rmk}
		Here and in what follows in this section, we will denote the sheaves on the moduli space by its fibers over closed points.
		So $R \Hom_S(E_i, E_i)$ stands for $R \Hom_S(\CE_i, \CE_i)$ where $\CE_i$ is the $i$-th summand in the decomposition of the universal sheaf $\CE|_{M_{\mathrm{ext}} \times S \times \BA^1}$
		under the decomposition \eqref{decompo}. As before we write $R \Hom_S(-,-) = \pi_{\ast} \hom(-, -)$ where $\pi$ is the projection away from $S$.
	\end{rmk}
	\begin{rmk}
		The right hand side of \eqref{fix pot} is the 
		%The perfect-obstruction theory on the right hand sideabove identifies 
		%the fixed perfect obstruction theory with the natural one
		natural perfect obstruction theory
		appearing in the deformation theory of flags of sheaves $E_{\bullet}$.
		Indeed, assuming $i_0=0$ for simplicity and taking the long exact sequence in cohomology yields:
		\begin{multline*}
			\bigoplus_{i=0}^{r-1} \Hom(E_i, E_{i+1}) / \BC f_i \to T_{M_{\mathrm{ext}},[E]}^{\vir}  \xrightarrow{\gamma} \bigoplus_{i=0}^{r} \Ext^1(E_i, E_i) \\
			\xrightarrow{\delta} \bigoplus_{i=0}^{r-1} \Ext^1(E_i, E_{i+1})_0 \to \Obs_{M_{\mathrm{ext}},[E]}^{\vir} \to \ldots  \ .
		\end{multline*}
		%Here we have written $T_{M_{\mathrm{ext}},[E]}^{\vir}, \Obs_{M_{\mathrm{ext}},[E]}^{\vir}$ for the virtual tangent and obstruction sheaf
		%of $M_{\mathrm{ext}}$, i.e. the $h^0,h^1$ of $\left( T^{\text{vir}}|_{M_{\mathrm{ext}}} \right)^{\textup{fixed}}$.
		The first term parametrize deformations of the maps $f_i : E_i \to E_{i+1}$.
		The map $\gamma$ sends a deformation of the flag $E_{\bullet}$ to the deformation of the individual terms $E_i$ in the flag.
		Given a deformation of the individual terms $(\alpha_i) \in \oplus_i \Ext^1(E_i, E_i)$,
		its image under $\delta$ vanishes if and only if the diagram
		\[
		\begin{tikzcd}
			E_i \ar{r}{f_i} \ar{d}{\alpha_i} & E_{i+1} \ar{d}{\alpha_{i+1}} \\
			E_i[1] \ar{r}{f_i} & E_{i+1}[1]
		\end{tikzcd}
		\]
		commutes, hence if and only if the deformations are compatible with $f_i$.
		We refer to \cite{GT1} for a discussion.
	\end{rmk}
	\begin{proof}
		We linearize the line bundle $\CO_{\p^1}(-1)$
		such that it has weight $\Ft$ over $0 \in \p^1$ (and hence weight $0$ over $\infty \in \p^1$).
		%Let $E \in M_{\mathrm{ext}}$. 
		By replacing $E$ by $E(k)$ for appropriate $k$ (and raising $r$ if needed) we can assume that
		\[ E|_{S \times \BA^1} = E_{0} \oplus E_{1} \Ft^{1} \oplus \cdots \oplus E_{r-1} \Ft^{r-1} \otimes F \Ft^{r} \oplus F \Ft^{r+1} \oplus F \Ft^{r+2} \oplus \ldots  \,. \]
		for the restriction of the universal family to $M_{\mathrm{ext}}$.
		
		Let $\iota : S \to X := S \times \p^1$ be the inclusion of the fiber over $0$.
		We argue now similarly to \cite[Prop.3.12]{GSY}.
		The idea is to peel off one factor of $E_i$ at a time. Concretely,
		define a sequence of sheaves $E^{(j)}$ inductively by
		$E^{(0)} := E$ and by the short exact sequence
		\begin{equation} 0 \to E^{(j+1)}(-1) \to E^{(j)} \to \iota_{\ast} E_{j} \to 0. \label{induction} \end{equation}
		Since the flag $E_{\bullet}$ stabilizes at the $r$-th step, we have 
		\[ E^{(r)} = \pi_S^{\ast}(F). \]
		
		\begin{lemma} \label{lemma:E peel 1} For any $A \in D^b(S)$ we have
			\[
			R\Hom_X(E^{(j)}, \iota_{\ast}A)
			=
			R\Hom_S(E_j, A) + \sum_{k \geq 1} \Ft^{-k} \otimes R \Hom_S( E_{j+k} - E_{j+k-1}, A ).
			\]
		\end{lemma}
		\begin{proof}
			Since $E_{\bullet}$ stabilizes the sum on the right hand side has only finitely many non-zero terms, so the claim is well-defined.
			We argue by induction. First apply $R \Hom( -, \iota_{\ast} A)$ to the sequence \eqref{induction},
			then we use adjunction with respect to $\iota$ and the well-known fact (e.g. \cite[Cor.11.4]{Huy}) that for any $B \in D^b(S)$ we have the distinguished triangle
			\[ B(-S_0)[1] \to L \iota^{\ast} \iota_{\ast} B \to B \to B(-S_0)[2]. \]
			This yields:
			\[ R\Hom_X(E^{(j)}, \iota_{\ast}A)
			=
			R \Hom_S(E_j, A) - R \Hom_S(E_j, A) \otimes \Ft^{-1} + \Ft^{-1} \otimes R\Hom_X(E^{(j+1)}, \iota_{\ast}A) \]
			from which the claim follows by induction.
		\end{proof}
		
		\begin{lemma} \label{lemma:E peel 2} For any $A \in D^b(S)$ and $j$ we have
			\[ R \Hom_X( \iota_{\ast} A, E^{(j)}) = ( R \Hom_X( E^{(j)}, \iota_{\ast} A[3]) \otimes \Ft )^{\vee}. \]
		\end{lemma}
		\begin{proof}
			By Serre duality we have
			\[ R \Hom_X( \iota_{\ast} A, E^{(j)}) = R \Hom_X( E^{(j)}, \iota_{\ast} A \otimes \omega_X[3])^{\vee}
			= (R \Hom_X( E^{(j)}, \iota_{\ast} A[3]) \otimes \Ft)^{\vee}, \]
			where we used that $\omega_X|_{S_0} = \Omega_{\p^1,0} \otimes \CO_S = \Ft \otimes \CO_S$.
		\end{proof}
		
		\begin{lemma} \label{lemma:fix}
			\[ R \Hom_X(E^{(j)}, E^{(j)} )^{\textup{fixed}}[1]
			\cong \mathrm{Cone}\left( \bigoplus_{i=j}^{r} R \Hom_S(E_i, E_i) \xrightarrow{\delta} \bigoplus_{i=j}^{r-1} R \Hom_S(E_i, E_{i+1}) \right). \]
		\end{lemma}
		\begin{proof}
			By applying $R \Hom(E^{(j)}, -)$ to \eqref{induction} we obtain the distinguished triangle
			\[ R\Hom_X(E^{(j)}, \iota_{\ast}E_j) \to R\Hom_X(E^{(j)}, E^{(j+1)}(-1))[1] \to R\Hom_X(E^{(j)},E^{(j)})[1]. \]
			By Lemma~\ref{lemma:E peel 1} we have
			\[ R\Hom_X(E^{(j)}, \iota_{\ast}E_j)^{\textup{fixed}} = R\Hom_S( E_j, E_j ). \]
			For the second term we apply $R \Hom(-, E^{(j+1)}(-1))$ to \eqref{induction} and obtain the 
			distinguished triangle
			\[ R\Hom_X(E^{(j+1)}, E^{(j+1)}) \to R \Hom_X(\iota_{\ast} E_j, E^{(j+1)}(-1))[1] \to R\Hom_X(E^{(j)}, E^{(j+1)}(-1))[1]. \]
			
			By Lemma~\ref{lemma:E peel 2} we have
			\begin{align*}
				R \Hom_X(\iota_{\ast} E_j, E^{(j+1)}(-1))
				& = R \Hom_X(\iota_{\ast} E_j(1), E^{(j+1)}) \\
				& = R \Hom_X(\iota_{\ast} E_j, E^{(j+1)}) \otimes \Ft \\
				& = R \Hom_X( E^{(j+1)}, \iota_{\ast} E_j[3])^{\vee}.
			\end{align*}
			Taking the fixed part, using Lemma~\ref{lemma:E peel 1} and Serre duality on $S$ we get
			\[
			R \Hom_X(\iota_{\ast} E_j, E^{(j+1)}(-1))^{\textup{fixed}}[1] = R \Hom(E_j, E_{j+1}).
			\]
			Taking both statements together we end up with the distinguished triangle:
			\begin{multline*}
				R\Hom_S( E_j, E_j ) \oplus R\Hom_X(E^{(j+1)}, E^{(j+1)})^{\text{fixed}} \\
				\to R \Hom(E_j, E_{j+1}) \to R\Hom_X(E^{(j)},E^{(j)})^{\text{fixed}}[1].
			\end{multline*}
			
			The last piece of information we need is that
			\[ R \Hom_X( E^{(r)}, E^{(r)} ) = %= R \Hom_X( E^{(r+1)}, E^{(r+1)} ) = \ldots = 
			R \Hom_S( F, \pi_{\ast} \pi^{\ast}(F)) = \Hom_S(F,F). \]
			Hence by iterating the above argument, the claim now follows by induction,
			see also \cite{GSY, GT1, GT2} for a discussion of the maps.
		\end{proof}
		
		\begin{lemma} \label{lemma:mov} In $K$-theory we have:
			\[ R \Hom_X(E^{(j)}, E^{(j)} )^{\textup{mov}}[1]
			=
			\sum_{i \geq j} \sum_{k \geq 1} 
			\left( \begin{array}{c}
				- \Ft^{-k} \otimes R\Hom_S( E_{i+k} - E_{i+k-1}, E_i) \\
				+ \Ft^{k} \otimes R\Hom_S( E_{i+k+1} - E_{i+k}, E_i)^{\vee}
			\end{array} \right).
			\]
		\end{lemma}
		\begin{proof}
			By the same argument as in Lemma~\ref{lemma:fix} but now taking the moving part.
		\end{proof}
		
		We complete the proof of Proposition~\ref{prop:fixed virtual class}.
		The first part follows from Lemma~\ref{lemma:fix},
		taking the tracefree part and arguing as in \cite[Proof Thm.7.1]{GT1}.
		The second part follows directly from Lemma~\ref{lemma:mov} by taking $j=0$ (since the trance part $R \Gamma(X,\CO_X) = R \Gamma(S, \CO_S)$ is $\BC^{\ast}$-fixed).
		%%where the maps are explained in \cite{GSY} (see also \cite{GT1, GT2})
	\end{proof}

	\subsection{Second cosection}
	
	\begin{prop} \label{prop:double cosection}
		%Assume that $\rk(v) = 1$. 
		Let $N \subset M_{\mathrm{ext}}$ be a connected component and let
		$s = | \{ i : E_i \neq E_{i+1} \} |$
		be the number of non-trivial steps in the flag $E_{\bullet}$.
		If $s \geq 2$, then
		the virtual class of the fixed perfect-obstruction theory on $N$ vanishes:
		\[ [N]^{\vir} = 0. \]
	\end{prop}
	\begin{proof}
		The strategy is to construct a second, linearly independent cosection of the fixed obstruction theory of the component $N$.
		Since there is then a second trivial piece in the obstruction theory, the virtual class vanishes: $[N]^{\text{vir}} = 0$.
		This idea was pioneered in the proof of the Katz-Klemm-Vafa conjecture by Pandharipande and Thomas in \cite{PT}.
		It can be viewed as the source of all multiple cover behaviour on the sheaf side.
		
		In rank $1$ the existence of a cosection can be seen quite easily.
		Indeed, if $\rk(v)=1$, we can assume that $v$ is the Mukai vector of the ideal sheaf of some length zero subscheme,
		so that $E_r \subset \CO_S$ is an ideal sheaf, and hence all $E_i$ are ideal sheaves of some curves.
		By definition (see \eqref{defn eiej0}) we have the exact sequence
		\begin{equation} \label{upps}
			\begin{gathered}
				%\Hom(E_i, E_{i+1})/ \BC f_i \to 
				0 \to \Ext^1(E_i, E_{i+1}) \to \Ext^1(E_i, E_{i+1})_{0} \\
				\xrightarrow{h_i} H^2(S,\CO_S) \to \Ext^2(E_i, E_{i+1}) \to \Ext^2(E_i, E_{i+1})_{0} \to 0.
			\end{gathered}
		\end{equation}
		%Since $E_i \subset E_r$, all $E_i$ are ideal sheaves. 
		Since the $E_i$ are stable one has that
		\[ \Ext^2(E_i, E_{i+1}) = \Hom(E_{i+1}, E_i)^{\vee}
		=
		\begin{cases}
			0 & \text{ if } E_i \neq E_{i+1} \\
			\BC & \text{ if } E_i = E_{i+1}.
		\end{cases}.\]
		If $E_i \neq E_{i+1}$ we hence find that $h_i$ is surjective.
		Consider the diagram:
		\[
		\begin{tikzcd}
			\bigoplus_{i=0}^{r} R \Hom_S(E_i, E_i)_0 \ar{r}{\delta} & \bigoplus_{i=0}^{r-1} R \Hom_S(E_i, E_{i+1})_{0} \ar{r} \ar{d}{h} &
			h^1( T_{N}^{\text{vir}} ) \ar{r} \ar[dotted]{dl}{ \exists}& 0 \\
			& \CO^{\oplus s} & .
		\end{tikzcd}
		\]
		where $h = \oplus_{i: E_i \neq E_{i+1}} (h_i \circ \mathrm{pr}_i)$ and $\mathrm{pr}_i$ is the projection to the $i$-th summand.
		By the previous argument $h$ is surjective.
		By the claim below, the map above factors through a surjection
		$h^1( T_{N}^{\text{vir}} ) \to \CO^s$.
		We see that the obstruction sheaf has $s$ trivial summands.
		The reduced obstruction theory removes one of these summands.
		Hence if $s \geq 2$, there is a positive number of trivial summands and the virtual class vanishes.
		
		\vspace{3pt}
		\noindent
		\textbf{Claim.} The composition $h \circ \delta$ vanishes.
		
		\begin{proof}[Proof of Claim]
			Given $(\alpha_i) \in \oplus_i \Ext^1(E_i, E_i)$ we need to show that $h( \alpha_{i+1} \circ f_i - f_i \circ \alpha_i )$
			vanishes.
			By \eqref{defn eiej0},
			$f_i \circ \alpha$ lies in the image of the map
			$\Ext^1(E_i, E_{i+1}) \to \Ext^1(E_i, E_{i+1})_{0}$,
			which using \eqref{upps} shows that $h_i(f_i \circ \alpha) = 0$.
			Similarly, given $\alpha_{i+1} \in \Ext^1(E_{i+1}, E_{i+1})$ the composition $\alpha_{i+1} \circ f_i$ lies in $\Ext^1(E_i, E_{i+1})$.
		\end{proof}
		
		In higher rank the above naive approach does not work since
		the individual sheaves $E_i$ can behave quite badly: they do not have to be semi-stable, and there can be maps $E_{i+1} \to E_i$.
		Instead, we will imitate the arguments of \cite{PT}. If $E \in M_{\mathrm{ext}}$, there exists
		a canonical injection $E \hookrightarrow \pi_S^{\ast}(F)$. Let $Q$ be the cokernel. We can hence view $E$ as the 'stable pair':
		\[ E = [ \pi_S^{\ast}(F) \xrightarrow{\varphi} Q ]. \]
		The cokernel $Q$ is supported over $S \times \Spec( \BC[x]/(x^r) )$ and is identified there with
		\[
		Q = F/E_{0} \oplus F/E_{1} \Ft^{1} \oplus \cdots \oplus F/E_{r-1} \Ft^{r-1} \,.
		\]
		Then applying the construction of \cite[Sec.5.4]{PT} to $Q$
		(and deforming the section $\varphi$ as in Sec.5.5 of \emph{loc.cit.})
		yields an explicit first-order deformation of $Q$ of weight\footnote{In \cite{PT} the vector field has weight $1$.
			We have the sign difference to our case because we let out torus act with tangent weight $-1$ at $0 \in \BA^1$ whereas in \cite{PT} it is with weight $1$.}
		$-1$.
		Arguing as in \cite[Prop.12]{PT} shows that this vector field is linearly independent from the translational shift (that gives the reduced obstruction theory),
		if and only if $Q$ is uniformly $r$-times thickened,
		hence if and only if $E_0 = \ldots = E_{r-1}$, hence only if there is at most $1$ step.
		By Serre duality we hence obtain the second independent cosection whenever $s \geq 2$.
	\end{proof}

	\subsection{The contributions from the single-step locus}
	\label{subsec:single step}
	We consider the components $M^{1\text{-step}}_r \subset M_{\mathrm{ext}}$ with a single step
	and a uniform thickened sheaf of length $r$,
	i.e. the component parametrizing sheaves $E$ of the form
	\begin{equation} \label{KF}
		E|_{S \times \BA^1} = K\Ft^{i_0} \oplus K \Ft^{i_0+1} \oplus \cdots \oplus K \Ft^{i_0 + r-1} \oplus F \Ft^{i_0 + r} \oplus F \Ft^{i_0 + r+1} \oplus \ldots  \,.
	\end{equation}
	We have the (non-equivariant) exact sequence
	\[ 0 \to \pi_S^{\ast}(F)\otimes \CO_{\p^1}(-r - i_0) \to E|_{S \times \BA^1} \to \iota_{r \ast} \pi_S^{\ast}(K) \to 0, \]
	where $\iota_r : S \times \Spec(k[x]/x^r) \to S \times \p^1$ is the inclusion.
	%and $\pi_S$ is the projection to $S$.
	We obtain that 
	\[ w = r v(K) - (r + i_0) v. \]
	By taking the pairing with $v(y)$ and using $w \cdot v(y) = 0$ (see \eqref{dersdf}) we find that
	$0 = r v(K) \cdot v(y) - (r + i_0)$, and hence that
	$i_0 = s_r \cdot r$ for some $s_r \in \BZ$.
	This shows that $w$ is divisibile by $r$ as well.
%	
%	
%	After wedging with $v$ we find that $r | \div(v \wedge w)$
%	and by our assumption that $\div(v \wedge w) = \div(w)$ therefore $r | w$.
%	Inserting in the equation above, we get $i_0 = s_r \cdot r$ for some $s_r \in \BZ$. 
	We conclude that:
	\[ v(F) = v, \quad v(F/K) = -\frac{w}{r} - s_r v =: u \]
	Moreover, since $F/K$ is a quotient of the torsion-free $F$ it has to be of rank in the interval $[0, \rk(v)-1]$.
	We see that $s_r$ is the unique integer such that
	%\[ 0 \leq  -\frac{\rk(w)}{r} - s_r \rk(v) \leq \rk(v)-1. \]
\begin{equation} -\frac{\rk(w)}{r} - s_r \rk(v) \in [0, \rk(v)-1]. \label{dfsd0-333} \end{equation}
	
	The inclusion $K \subset F$ defines an element in the Quot scheme of $F$
	with quotients of Mukai vector $u$.
	Let $\CF \in \Coh( M(v) \times S )$ be the universal sheaf over the moduli space.
	We conclude that the component is the relative Quot scheme:
	\[ M^{1\text{-step}}_r = \Quot( \CF / M(v), u ). \]
	
	By Proposition~\ref{prop:fixed virtual class} we have the virtual normal bundle
	\begin{align*}
		N^{\text{vir}}
		& =
		- R \Hom_S(K, F-K)^{\vee} \otimes (\Ft^{-1} + \ldots + \Ft^{-r})\\
		& \phantom{=} + R \Hom_S(K, F-K) \otimes (\Ft + \ldots + \Ft^{r-1})
	\end{align*}
	Moreover, we have
	\begin{align*}
		\rk \Hom_S(K,F-K) 
		& = -v(K) \cdot v(F/K) \\
		& = v(F/K)^2 - v(F) \cdot v(F/K) \\
		& = u^2 - u \cdot v \\
		& \equiv u \cdot v \quad (\text{modulo } 2) \\
		& \equiv \frac{w \cdot v}{r} \quad (\text{modulo } 2).
	\end{align*}
	We find that the contribution of 
	$M^{1\text{-step}}_r$ to the
	virtual class of $M_{v,w}(S \times \p^1/S_{\infty})$
	in the localization formula is given by:
	\begin{equation} \label{EXTREMECONTRIBUTION}
		\begin{aligned}
			& \frac{1}{e_{\BC^{\ast}}(N^{\vir})} \cdot [ M^{1\text{-step}}_r ]^{\text{vir}} \\
			& =
			\frac{e_{\BC^{\ast}}\left(R \Hom(K, F-K)^{\vee} \otimes (\Ft^{-1} + \ldots + \Ft^{-(r-1)})\right)}{ e_{\BC^{\ast}}\left( \left( R \Hom(K, F-K) \otimes (\Ft + \ldots + \Ft^{r-1}) \right) \right) } \cdot \\
			& \quad \quad \quad \quad 
			\cdot e_{\BC^{\ast}}\left( R \Hom_S(K, F-K)^{\vee} \otimes \Ft^{-r} \right) \cdot
			[ \Quot( \CF/M(v), u ) ]^{\text{vir}} \\
			& =
			(-1)^{(r-1) \cdot \frac{w \cdot v}{r}}
			e_{\BC^{\ast}}\left( R \Hom_S(K, F-K)^{\vee} \otimes \Ft^{-r} \right) \cdot
			[ \Quot( \CF/M(v), u ) ]^{\text{vir}} \\
			& =
			(-1)^{(r-1) \cdot \frac{w \cdot v}{r}}
			\sum_{k \in \BZ} (-rt)^{-k} c_{k+u \cdot (u-v)}( R \Hom_S(K, F-K)^{\vee} ) \cdot [ \Quot( \CF/M(v), u ) ]^{\text{vir}}
		\end{aligned}
	\end{equation}
	where we used that $e_{\BC^{\ast}}(V) = (-1)^{\rk(V)} e_{\BC^{\ast}}(V^{\vee})$.

	\subsection{Contributions from mixed and pure-rubber components}
	\label{subsec:mixed pure rubber}
	Aside from the extremal locus $M_{\mathrm{ext}}$, there are two other types of components in the fixed locus
	$M_{v,w}(S \times \p^1/S_{\infty})^{\BC^{\ast}}$:
	\begin{itemize}
		\item[(i)] Components $W$ of the fixed locus which parametrize sheaves supported on an expanded degeneration of $S \times \p^1$,
		but whose restriction to $S \times \p^1|_{\BA^1}$ is \emph{not} the pullback of a sheaf from $S$. We call these components $W$ of {\em mixed type}.
		These components are given as a fiber product over $M(v)$ of an extremal component $M_{\mathrm{ext}}'$ (for different Chern characters) and a rubber component $M_{v',w'}(S \times \p^1 / S_{0,\infty})^{\sim}$.
		The perfect obstruction theory decomposes according to this decomposition and one observes that each summand admits a surjective cosection.
		Since the reduced class is formed by removing one trivial $\CO$-summand,
		one trivial summand remains and the fixed virtual class vanishes.
		Hence the contribution from this component vanishes.
		\item[(ii)] The second type of component parametrizes sheaves on an expanded degeneration but whose restriction to $S \times \p^1$ is pulled back from $S$.
		We call it of {\em pure rubber type}. It is given by the rubber space
		\[ W = M_{v,w}(S \times \p^1 / S_{0,\infty})^{\sim}. \]
		Because of the $\BC^{\ast}$-scaling action, the {\em reduced} virtual class of the fixed obstruction theory is of dimension $v^2 + 2 = 2n$. 
		The virtual normal bundle is $N^{\text{vir}} = T_{\p^1, \infty} \otimes \CL_{0}^{\vee}$
		where $\CL_{0} \to W$ is the relative cotangent line bundle over the rubber space (of the marking glued to $\p^1$).
		Hence one finds the contribution
		\begin{equation}
		\label{123purerubber}
		\frac{1}{e_{\BC^{\ast}}(N^{\text{vir}})} [ W ]^{\text{vir}}
		=
		\frac{1}{t - c_1(\CL_0)} [ M_{v,w}(S \times \p^1 / S_{0,\infty})^{\sim} ]^{\text{vir}}.
		\end{equation}
		Since the reduced virtual dimension of $[W]^{\text{vir}}$ (which is $2n$) is strictly smaller than that of $M_{v,w}(S \times \p^1/S_{\infty})^{\BC^{\ast}}$ (which is $2n+1$), we will see below that
		the contribution from $[W]^{\text{vir}}$ to suitable integrals over the full moduli space will vanish.
	\end{itemize}

\subsection{Quasimap wall-crossing}
We turn to the proof of the quasimap wall-crossing (Theorem \ref{thm:wallcrossing higher rank}):

Consider the pure rubber component
$M_{v,w}(S \times \p^1 / S_{0,\infty})^{\sim}$
and the evaluation map
\[ \ev_{\infty} : M_{v,w}(S \times \p^1 / S_{0,\infty})^{\sim}
\to M(v) \]
given by intersection with the fiber over $\infty \in \p^1$.
Let $\mathrm{pt} \in H^{4n}(M(v))$ be the class of a point and consider the integral:
\[
\langle \mathrm{pt}, 1 \rangle^{S \times \p^1 / S_{0,\infty}}_{v,w}
:=
\int_{[ M_{v,w}(S \times \p^1 / S_{0,\infty})^{\sim} ]^{\vir}}
\ev_{\infty}^{\ast}(\mathrm{pt}).
\]

We have the following wall-crossing formula proven by Nesterov:
\begin{thm}[{\cite[Theorem 3.5]{N2}}] \label{thm:Nesterov}
		\[
		\DT^{S \times E}_{(v,w)}
		=
		\GW^{M(v)}_{E, w'}
		+
		\chi(S^{[n]})
\langle \mathrm{pt}, 1 \rangle^{S \times \p^1 / S_{0,\infty}}_{v,w}
		\]
		where $w' = - \langle w, - \rangle : v^{\perp} \to \BZ$ is the homology class induced by $w$,
	\end{thm}

We prove here the following evaluation of the wall-crossing term:
%have the following evaluation:
\begin{thm} \label{thm:wc term} For any $F \in M(v)$ we have
\[ \langle \mathrm{pt}, 1 \rangle^{S \times \p^1 / S_{0,\infty}}_{v,w}
=
- \sum_{r | w}
		\frac{ (-1)^{w \cdot v} }{r}
		e^{\textup{vir}}(\Quot(F,u_r))
		\]
		where $u_r= -w/r - s_r v$ for the unique $s_r \in \BZ$ such that $0 \leq \rk(u_r) \leq \rk(v)-1$
\end{thm}
\begin{proof}
The moduli space $M_{v,w}(S \times \p^1/S_{\infty})$ is proper and has a (reduced) virtual fundamental class of dimension $2n+1$.
Hence for dimension reasons we have
\[
\int_{[ M_{v,w}(S \times \p^1/S_{\infty}) ]^{\vir}} \ev_{\infty}^{\ast}(\mathrm{pt}) = 0.
\]
We lift the virtual class to an equivariant class and apply the virtual localization formula. By the discussion in Section~\ref{subsec:mixed pure rubber}, only the extremal component and the component of pure rubber type contributes in the localization formula, which gives:
\[
\int_{[ M_{\mathrm{ext}} ]^{\text{red}} } \frac{\ev_{\infty}^{\ast}(\mathrm{pt})}{e_{\BC^{\ast}}( N^\vir )}
+
\int_{[ M_{v,w}(S \times \p^1 / S_{0,\infty})^{\sim} ]^{\vir}}
\frac{\ev_{\infty}^{\ast}(\mathrm{pt})}{t - c_1(\CL_0)} = 0
\]
Taking the $t^{-1}$-coefficient we obtain
\[
\langle \mathrm{pt}, 1 \rangle^{S \times \p^1 / S_{0,\infty}}_{v,w} =
(-1) \cdot \mathrm{Coefficient}_{t^{-1}}\left[ \int_{[ M_{\mathrm{ext}} ]^{\text{red}} } \frac{\ev_{\infty}^{\ast}(\mathrm{pt})}{e_{\BC^{\ast}}( N^\vir )} \right].
\]
For any $F \in M(v)$ we have
		\begin{align*}
			& \int_{[ M_{\mathrm{ext}} ]^{\text{vir}} } \frac{1}{e_{\BC^{\ast}}( N^\vir )} \ev_{\infty}^{\ast}(\mathrm{pt}) \\
			\overset{(\ast)}{=} & \sum_{r | w} \int_{[ M^{1\text{-step}}_r ]^{\text{vir}}} \frac{1}{e_{\BC^{\ast}}(N^{\text{vir}})} \ev_{\infty}^{\ast}(\mathrm{pt}) \\
			\overset{(\ast \ast)}{=} & \sum_{r|w} (-1)^{(r-1) \cdot \frac{w \cdot v}{r}}
			\sum_{k \in \BZ} (-rt)^{-k} \int_{ [ \Quot( \CF/M(v), u_r ) ]^{\text{vir}} } c_{k+u_r \cdot (u_r-v)}( R \Hom_S(K, F-K)^{\vee} ) \rho^{\ast}(\mathrm{pt}) \\
			\overset{(\ast \ast \ast)}{=} & \sum_{r|w} \frac{1}{-r t} (-1)^{(r-1) \frac{w \cdot v}{r}} (-1)^{1 + u_r \cdot (u_r-v)}\int_{[ \Quot( F, u_r ) ]^{\text{vir}}} c_{1+u_r \cdot (u_r-v)}\left( R \Hom( K, F- K) + \CO \right) \\
			= & \sum_{r|w} \frac{1}{r t} (-1)^{w \cdot v}  e^{\text{vir}}( \Quot(F,u_r) ).
		\end{align*}
		where $(\ast)$ follows since only the $1$-step locus contributes by
		Proposition~\ref{prop:double cosection},
		$(\ast \ast)$ follows from \eqref{EXTREMECONTRIBUTION} and by checking that the $K$-theory class of the fixed obstruction theory
		given in Proposition \ref{prop:fixed virtual class} equals the class of the
		reduced perfect obstruction theory  of the relative Quot scheme $\Quot( \CF / M(v))$
		and we let $\rho : \Quot( \CF/M(v)) \to M(v)$ denote the morphism to the base,
		and $(\ast \ast \ast)$ follows since $\Quot(F,u_r)$ has reduced virtual dimension $u_r \cdot (u_r-v)+1$, see Section~\ref{subsec:quot k3}.
		This yields the claim.
%
%
%and using Theorem~\ref{thm:wallcrossing term for real}
%yields the claim.
%
%\langle \mathrm{pt}, 1 \rangle^{S \times \p^1 / S_{0,\infty}}_{v,w}
%=
%
%- t \int_{[ M_{\mathrm{ext}} ]^{\text{red}} } \frac{1}{e_{\BC^{\ast}}( N^\vir )} \ev_{\infty}^{\ast} [F] \]
\end{proof}

%We have obtained the proof of our first main theorem:
\begin{proof}[Proof of Theorem~\ref{thm:wallcrossing higher rank}]
This follows by combining Theorem~\ref{thm:Nesterov}
and Theorem~\ref{thm:wc term}.
\end{proof}
%\mathrm{pt}
%%\end{cor}
%On the pure rubber type components
%Consider the evaluation morphism given by intersection with the fiber over 
%be the evaluation maps on the pure-rubber components.

%	Theorems~\ref{thm:Wall  cross rank 1} and 

\subsection{Rank 1} \label{subsec:cap rank 1}
Let $n \geq 0$ and assume that
\[ v = (1,0,1-n). \]
We fix the class $y=-[\CO_{s}]$ for a point $s \in S$ which satisfies $v \cdot v(y)=1$. The condition $w \cdot v(y) = 0$ then says that $w$ is of rank zero, so it is of the form
%After tensoring with a line bundle from $\p^1$ we can assume that
	\[ w = (0, -\beta, -m) \]
for some $\beta \in \Pic(S)$ and $m \in \BZ$.
%	We have $\div(v \wedge w) = \div(\beta,m)$,

%For divisors $r | (\beta,m)$ the condition \eqref{dfsd0-333} shows that $s_r = 0$, so $i_0=0$.

We prove the quasimap wall-crossing in rank $1$:
\begin{proof}[Proof of Theorem~\ref{thm:Wall  cross rank 1}]
This follows by specializing Theorem~\ref{thm:wallcrossing higher rank} to $v=(1,0,1-n)$ and $w=(0,-\beta,-m)$. Indeed,
in this case the equality of moduli spaces
\[ M_{(v,w)}(S \times E) = \Hilb_m(S \times E, (\beta,n)) \]
shows 
\[ \DT^{S \times E}_{(v,w)} = \DT^{S \times E}_{m, (\beta,n)}. \]

Moreover, the homology class induced by $w$,
\[ w' = -\langle w, - \rangle  \in (v^{\perp})^{\ast} \cong H_2(S^{[n]},\BZ) \]
is precisely $\beta + m A$, see for example \cite[Example 2.3]{Universality}.

Finally, for divisors $r | (\beta,m)$ the condition \eqref{dfsd0-333} shows that $s_r = 0$, so $i_0=0$, and hence $u_r = \frac{1}{r}(0,\beta, m)$
in Theorem~\ref{thm:wallcrossing higher rank}.
Inserting everything we obtain
 		\[
		\DT^{S \times E}_{m, (\beta,n)}
		=
		\GW^{S^{[n]}}_{E, \beta,m}
		-
		\chi(S^{[n]})
		\sum_{r | (\beta,m)} \frac{1}{r} (-1)^{m}
		e^{\textup{vir}}(\Quot(I_z, \frac{ (0,\beta,m) }{r} ) )
		\]
where $I_z$ is the ideal sheaf of a subscheme $z \in S^{[n]}$,
which is precisely what was claimed.
%This is precisely the claim.
\end{proof}

\begin{rmk} \label{rmk:1 step component in rank 1}
For later use we also remark that in rank $1$ the $1$-step component
can be written as the relative Hilbert scheme
	\[ M_{r}^{1\text{-step}}
	=
	\Quot( \CI_{\CZ} / S^{[n]}, (0, \beta/r, m/r) )
	=
	S_{\beta/r}^{[n_1, n]} \]
	where $n_1 = n + \frac{m}{r} + \frac{1}{2} \beta^2/r^2$.
\end{rmk}

%Combining Theorem~\ref{thm:Nesterov} and Corollary~\ref{cor:wc term}
%in this case we hence get:
%\[
%		\DT^{S \times E}_{(1,0,1-n), (0,-)}
%		=
%		\GW^{M(v)}_{E, w'}
%		+
%		\chi(S^{[n]})
%\langle \mathrm{pt}, 1 \rangle^{S \times \p^1 / S_{0,\infty}}_{v,w}
%		\]
%		where $w' = - \langle w, - \rangle : v^{\perp} \to \BZ$ is the homology class induced by $w$,
%
%
%
%		\[
%		\DT^{S \times E}_{m, (\beta,n)}
%		=
%		\GW^{S^{[n]}}_{E, \beta,m}
%		-
%		\chi(S^{[n]})
%		\sum_{r | (\beta,m)} \frac{1}{r} (-1)^{m}
%		\mathsf{Q}_{n,\frac{1}{r}(\beta,m)}
%		\]
%		
%		
%We obtain:
%
%
%%For divisors $r | (\beta,m)$ the sheaves parametrized by $M_{r}^{1\text{-step}}$ 	are all of the form 
%%	\eqref{KF} where $i_0 = 0$.
%	%have all $i_0$ (since $K$ is of rank $1$).
%	The $r$-step component can then be written as the relative Hilbert scheme
%	\[ M_{r}^{1\text{-step}}
%	=
%	\Quot( \CI_{\CZ} / S^{[n]}, (0, \beta/r, m/r) )
%	=
%	S_{\beta/r}^{[n_1, n]} \]
%	where $n_1 = n + \frac{m}{r} + \frac{1}{2} \beta^2/r^2$.
%	

%The contribution of the extremal locus computed in Theorem~\ref{thm:wallcrossing term for real} specializes to:
%\begin{cor}
%For any $z \in S^{[n]}$ with corresponding ideal sheaf $I_z$ we have
%\[ 
%		\int_{[ M_{\mathrm{ext}} ]^{\text{red}} } \frac{1}{e_{\BC^{\ast}}( N^\vir )} \ev_{\infty}^{\ast} [I_z]
%		\ =\  \sum_{r | (\beta,m)}
%		\frac{ (-1)^{m} }{rt}
%		e^{\textup{vir}}(\Quot(I_z, \frac{ (0,\beta,m) }{r} ) )
%		\]
%
%\end{cor}
		
%\[			\sum_{r | (\beta,m)} \frac{1}{r} (-1)^{m}
%		\mathsf{Q}_{n,\frac{1}{r}(\beta,m)}
%		\]

	\section{Multiple cover conjectures}
	\label{sec:multiple cover}
	
	\subsection{Definitions}
	Let $C$ be a smooth curve with distinct points $z_1, \ldots, z_k \in C$
	and consider the Hilbert scheme
	\[ \Hilb := \Hilb_{n,\beta,m}(S \times C/S_{z_1, \ldots, z_k}) \]
	parametrizing $1$-dimensional subschemes $Z$ of the relative geometry
	\[ (S \times C, S_{z_1, \ldots, z_k}), \quad S_{z_1, \ldots z_k} = \bigsqcup S \times \{ z_i \}. \]
	The numerical invariants are determined by the Mukai vector:
	\[ \ch(I_Z) \sqrt{\td}_S = (v(I_Z), w(I_Z)) = ( (1,0,1-n), (0, -\beta,-m) ). \]
	In other words,
	\[ [Z] = \iota_{\ast} \beta + n [C], \quad \ch_3(\CO_Z) = m, \]
	where $\iota : S = S \times \{ z \} \to S \times C$ is the natural inclusion for some $z \in C$.
	
	There exists evaluation maps at the markings
	\[ \ev_{z_i} : \Hilb_{n,\beta,m}(S \times C/S_{z_1, \ldots, z_k}) \to S^{[n]}. \]
	Let $\CZ \subset \Hilb \times S \times C$ denote the universal subscheme, and consider the diagram
	\[
	\begin{tikzcd}
		\Hilb \times S \times C \ar{r}{\pi_X} \ar{d}{\pi} & S \times C \\
		\Hilb.
	\end{tikzcd}
	\]
	For any $\gamma \in H^{\ast}(S \times C)$ we define the descendants
	\[ \tau_{i}(\gamma) = \pi_{\ast}( \ch_{2+i}(\CO_{\CZ}) \pi_X^{\ast}(\gamma) ) \in H^{\ast}(\Hilb). \]
	
	If $(\beta,m) \neq 0$, then the moduli space carries a reduced virtual class $[ \Hilb ]^{\vir}$ of dimension $(2-2g(C))n + 1$.
	Given $\lambda_1, \ldots, \lambda_k \in H^{\ast}(S^{[n]})$ and $\beta \neq 0$ we define
	\[ Z^{S \times C/S_{z_1, \ldots, z_k}}_{(\beta,n)}\left( \prod_{i} \tau_{k_i}(\gamma_i) \middle| \lambda_1, \ldots, \lambda_k \right)
	=
	\sum_{m \in \BZ} (-p)^{m}
	\int_{[ \Hilb ]^{\text{vir}} }
	\prod_{i} \tau_{k_i}(\gamma_i) \prod_{i} \ev_{z_i}^{\ast}(\lambda_i).
	\]
	%When clear we omit the target geometry from the supscript on the left.
	For any $(\beta,m)$ the moduli space $\Hilb$ carries also a standard (non-reduced) virtual class $[ \Hilb ]^{\std}$.
	By the existence of the non-trivial cosection it vanishes if $(\beta,m) \neq (0,0)$.
	If we integrate over the standard virtual class on the right hand side above,
	we decorate $Z$ with a superscript $\std$ on the left.
	In this case, $Z^{\std}( \ldots ) \in \BQ$.
	
	We state the degeneration formula \cite{LiWu} in the setting of reduced invariants \cite{MPT}.
	Related discussions appear in \cite{K3xE, K3xP1} and \cite[App.A]{HilbK3}.
	Let $C \rightsquigarrow C_1 \cup_x C_2$ be a degeneration of $C$.
	%and assume that $C_2$ is rational: $C_2 \cong \p^1$.
	Let
	\[ \{ 1, \ldots, k \} = A_1 \sqcup A_2 \]
	be a partition of the index set of points.
	We write $z(A_i) = \{ z_j | j \in A_i \}$.
	We choose that the points in $A_i$ specialize to the curve $C_i$.
	Fix also a partition of the index set of interior markings
	\[ \{ 1, \ldots, \ell \} = B_1 \sqcup B_2. \]
	Then for any $\alpha_i \in H^{\ast}(S)$ we have
	\begin{gather*}
		Z^{S \times C/S_{z_1, \ldots, z_k}}_{(\beta,n)}\left( \prod_{i=1}^{\ell} \tau_{k_i}( \omega \alpha_i ) \middle| \lambda_1, \ldots, \lambda_k \right) = \\
		Z^{S \times C_1/S_{z(A_1),x}}_{(\beta,n)}\left( \prod_{i \in B_1} \tau_{k_i}( \omega \alpha_i ) \middle| \prod_{i \in A_1} \lambda_i, \Delta_1 \right)
		Z^{S \times C_2/S_{z(A_2),x}, \std}_{(0,n)}\left( \prod_{i \in B_2} \tau_{k_i}( \omega \alpha_i ) \middle| \prod_{i \in A_2} \lambda_i, \Delta_2 \right) \\
		+
		Z^{S \times C_1/S_{z(A_1),x},\std}_{(0,n)}\left( \prod_{i \in B_1} \tau_{k_i}( \omega \alpha_i ) \middle| \prod_{i \in A_1} \lambda_i, \Delta_1 \right)
		Z^{S \times C_2/S_{z(A_2),x}}_{(\beta,n)}\left( \prod_{i \in B_2} \tau_{k_i}( \omega \alpha_i ) \middle| \prod_{i \in A_2} \lambda_i, \Delta_2 \right)
	\end{gather*}
	where $\Delta_1, \Delta_2$ stands for summing over the K\"unneth decomposition of the diagonal class
	\[ [ \Delta ] \in H^{\ast}(S^{[n]} \times S^{[n]}). \]
	
	Alternatively, we can also degenerate $C$ to an irreducible nodal curve and resolve it by a curve $C'$. Let $x_1, x_2 \in C'$ be the preimage of the node.
	Write $z= (z_1, \ldots, z_k)$.
	Then
	\begin{gather*}
		\label{deg_nodal}
		Z^{S \times C/S_{z}}_{(\beta,n)}\left( \prod_{i=1}^{\ell} \tau_{k_i}( \omega \alpha_i ) \middle| \lambda_1, \ldots, \lambda_k \right) 
		= 
		Z^{S \times C'/S_{z, x_1, x_2}}_{(\beta,n)}\left( \prod_{i=1}^{\ell} \tau_{k_i}( \omega \alpha_i ) \middle| \lambda_1, \ldots, \lambda_k, \Delta_1, \Delta_2 \right).
	\end{gather*}

	\begin{rmk}
		%We only highlight two points: 
		%First, 
		We index the generating series $Z$ by $\ch_3(\CO_Z)$ in order to avoid extra factors of $p$ in the degeneration formula above.
		%Indeed, for $Z \subset S \times (C_1 \cup C_2)$ with $Z_i = Z|_{S \times C_i}$ and 
		%$\eta = Z_1 \cap Z_2$ we have the exact sequence $0 \to \CO_Z \to \CO_{Z_1} \oplus \CO_{Z_2} \to \CO_{\eta} \to 0$. Hence
		%\[ \chi(\CO_{Z}) = \chi(\CO_{Z_1}) + \ch(\CO_{Z_2}) - \ell(\eta) \]
		%Using $m = \ch_3(\CO_Z)$ and $m_i = \ch_3(\CO_{Z_i})$ we find that
		%\[
		%m + (1-g) n = \chi(\CO_Z) %=  \chi(\CO_{Z_1}) + \ch(\CO_{Z_2}) - n 
		%= m_1 + (1-g_1) n + m_2 + (1-g_2) n - n = m_1 + m_2 + (1-g) n \]
		%where $g=g(C)$ and $g_i = g(C_i)$, and thus $m=m_1 +m_2$.
		%The case of the nodal degeneration is similar.
	\end{rmk}

	\subsection{Induction scheme}
	Our goal here is to express invariants of the form
	\begin{equation} Z^{S \times C/S_{z_1, \ldots, z_k}}_{(\beta,n)}(I | \lambda_1, \ldots, \lambda_k ),
		\quad I = \prod_{i=1}^{\ell} \tau_{k_i}( \omega \alpha_i )
		\label{Z general}
	\end{equation}
	%where $I = \prod_{i=1}^{\ell} \tau_{k_i}( \omega \alpha_i )$
	in terms of invariants of the cap geometry of the form
	\begin{equation} Z^{S \times \p^1/S_{\infty}}_{(\beta,n)}(I' | \lambda' ), \quad I' = \prod_{i} \tau_{k'_i}( \omega \alpha'_i ). \label{Z cap} \end{equation}
	The general strategy to do so is well-known and goes back 
	at least to work of Okounkov and Pandharipande in \cite{OPLocal} or even \cite{OkPandVir}.
	The process is in general highly non-linear. However for reduced theories as considered here the dependence becomes $\BQ$-linear. The linearity will allow us to prove the compatibility of the multiple cover formula with this process.
	
	This reduction is based on the following simple but useful lemma
	which relates descendants and the Nakajima basis of $S^{[n]}$.
	Given a cohomology weighted partition 
	\[ \mu = (\mu_i, \alpha_i)_{i=1}^{\ell}, \quad \mu_i \geq 0, \alpha_i \in H^{\ast}(S) \]
	define the special monomial of descendants
	\[ \tau[\mu] := \prod_{i=1}^{\ell} \tau_{\mu_i-1}( \alpha_i \cdot \omega ). \]
	%where we let $\omega \in H^2(C)$ denote the class of a point on $C$.
	and the Nakajima cycle
	\[ \mu := \prod_{i} \Fq_{\mu_i}(\gamma_i) 1_{S^{[0]}} \in H^{\ast}( S^{[ \sum_i \mu_i ]} ) \]
	where $\Fq_i(\alpha) : H^{\ast}(S^{[a]}) \to H^{\ast}(S^{[a+i]})$ are the Nakajima creation operators
	in the convention of \cite{NOY} (actually the precise convention is not important).
	Let $\CB$ be a basis of $H^{\ast}(S)$.
	We say $\mu$ is $\CB$-weighted if $\alpha_i \in \CB$ for all $i$.
	
	\begin{lemma}[{\cite[Proposition 6]{PaPixJap}}] \label{lemma:PaPixJap}
		The matrix indexed by $\CB$-weighted partitions of $n$ with coefficients
		\[ Z^{S \times \p^1/S_{\infty}, \std}_{(0,n)}( \tau[\mu], \nu ) \in \BQ \]
		is invertible.
	\end{lemma}
	\begin{proof}
		We give a sketch of the proof, see \cite{PaPixJap} for full details. For $(\beta,m)=(0,0)$, we have that
		\[ \Hilb_{n,(0,0)}(S \times \p^1/S_{z_1, \ldots, z_k}) = S^{[n]}. \]
		The virtual dimension matches the actual dimension, so the virtual class is just the fundamental class.
		Hence in this case one finds:
		\[
		Z^{S \times \p^1/S_{\infty}, z_1, \ldots, z_k, \std}_{(0,n)}( \tau[\mu] | \nu_1, \ldots, \nu_k ) 
		=
		\int_{S^{[n]}} \prod_{j=1}^{\ell(\mu)} \tau^{S^{n]}}_{\mu_j-1}(\alpha_j) \cdot \nu_1 \cdots \nu_k \]
		where $\tau^{S^{n]}}_{k}(\alpha) = \pi_{\ast}( \ch_{2+k}(\CO_{\CZ}) \pi_S^{\ast}(\alpha))$
		are the descendants on the Hilbert scheme $S^{[n]}$ (where $\CZ \subset S^{[n]} \times S$ is the universal family).
		The claim hence follows from checking that certain intersection numbers on the Hilbert scheme (determined by a suitable partial ordering)
		between descendants and Nakajima cycles do not vanish.
	\end{proof}
	
	\vspace{3pt}
	\noindent
	\textbf{The induction scheme:}
	We reduce the general invariants \eqref{Z general} to invariants \eqref{Z cap}
	by induction on the genus $g$ and the number of relative markings $k$.
	If $g(C) > 0$ we degenerate $C$ to a curve with a single node, and apply the degeneration formula in this case. %\eqref{deg_nodal}.
	If $g(C) = 0$ and $k \geq 2$, we consider the invariant
	\[
	Z^{S \times \p^1/S_{z_1, \ldots, z_{k-1}}}_{(\beta,n)}(I \cdot \tau[\lambda_k] | \lambda_1, \ldots, \lambda_{k-1} ).
	\]
	which is known by the induction hypothesis.
	The degeneration formula yields:
	\begin{gather*}
		Z^{S \times \p^1/S_{z_1, \ldots, z_{k-1}}}_{(\beta,n)}(I \tau[\lambda_k] | \lambda_1, \ldots, \lambda_{k-1} )
		= \\
		Z^{S \times \p^1/S_{z_1, \ldots, z_{k-1},x}}_{(\beta,n)}(I | \lambda_1, \ldots, \lambda_{k-1}, \Delta_1 ) Z^{S \times \p^1/S_{x}, \std}_{(0,n)}(\tau[\lambda_k] | \Delta_2 ) \\
		+ Z^{S \times \p^1/S_{z_1, \ldots, z_{k-1},x}, \std}_{(0,n)}(I | \lambda_1, \ldots, \lambda_{k-1}, \Delta_1 ) Z^{S \times \p^1/S_{x}}_{(\beta,n)}(\tau[\lambda_k] | \Delta_2 ).
	\end{gather*}
	By subtracting the second term on the right of the equality,
	and using Lemma~\ref{lemma:PaPixJap} to invert this relation, we see that
	\eqref{Z general} is a ($\BQ$-linear!) combination of terms which are lower in the ordering.
	Since the base case is $(g(C),k) = (0,1)$, this concludes the scheme.
	% (for us it is important that it is a $\BQ$-linear combination; this is what fails for non-reduced theories).

	\subsection{Statement and proof of multiple cover formula}
	Let $\beta \in \Pic(S)$ be an effective class
	and for every $r|\beta$ let $S_r$ be a K3 surface and
	\[ \varphi_r : H^{2}(S,\BR) \to H^{2}(S_r, \BR) \]
	be a real isometry such that
	$\varphi_r(\beta/r) \in H_2(S',\BZ)$
	is a primitive effective curve class.
	We extend $\varphi_r$ to the full cohomology lattice
	by $\varphi_r(\pt) = \pt$ and $\varphi_r(1) = 1$.
	We can further extend $\varphi_r$ to an action on the cohomology of the Hilbert scheme
	\[
	\varphi_r : H^{\ast}(S^{[n]}) \to H^{\ast}( S_r^{[n]} )
	\]
	by letting it act on Nakajima cycles by:
	\begin{equation}
		\label{induced on Hilb}
		\varphi_r( \prod_i \Fq_{\mu_i}(\alpha_i) 1 )
		=
		\prod_i \Fq_{\mu_i}( \varphi_r(\alpha_i)) 1.
	\end{equation}
	Then $\varphi_r$ is an isometric ring isomorphism and satisfies\footnote{This is clear if $\varphi_r:H^{\ast}(S) \to H^{\ast}(S_r)$ is the parallel transport operator of a deformation from $S$ to $S_r$, and it follows in general by observing that the parallel transport operators are Zariski dense in the space of isometries from $H^2(S,\BR) \to H^2(S_r,\BR)$.}
	\[ \varphi_r( \tau_i(\alpha) ) = \tau_i( \varphi_r(\alpha)). \]
	\begin{thm} \label{thm:mc for real}
		We have
		\begin{equation} \label{mc conjecture}
			\begin{gathered}
				Z^{S \times C/S_{z_1, \ldots, z_k}}_{(\beta,n)}\left( \prod_{i=1}^{\ell} \tau_{k_i}( \omega \alpha_i ) \middle| \lambda_1, \ldots, \lambda_k \right) \\
				=
				\sum_{r|\beta}
				Z^{S_r \times C/S_{z_1, \ldots, z_k}}_{(\varphi_r(\beta/r),n)}\left( \prod_{i=1}^{\ell} \tau_{k_i}( \omega \varphi_r(\alpha_i) ) \middle| \varphi_r(\lambda_1), \ldots, \varphi_r(\lambda_k) \right)(p^k)
			\end{gathered}
		\end{equation}
	\end{thm}
	\begin{proof}
		Since $\varphi_r$ is an isometry the morphism
		\[ \varphi_{r} \otimes \varphi_r : H^{\ast}(S^{[n]}) \otimes H^{\ast}(S^{[n]})
		\to H^{\ast}( S_r^{[n]} ) \otimes H^{\ast}( S_r^{[n]} ) \]
		sends $[ \Delta_{S^{[n]}} ]$ to $[ \Delta_{S_r^{[n]}} ]$.
		From this one shows in straightforward manner that \eqref{mc conjecture} is compatible with the degeneration formula.
		By the induction scheme of the previous section we are hence reduced to proving the statement for the cap 
		$S \times \p^1 / S_{\infty}$.
		
		Consider a class $\lambda \in H^{\ast}(S^{[n]})$ and 
		descendants $\tau_{k_i}(\omega \alpha_i)$, all homogeneous, such that
		\begin{equation} \label{deg condition}
			\deg(\lambda) + \sum_i \deg \tau_{k_i}(\omega \alpha_i) = 2n+1
		\end{equation}
		where $\deg(\gamma)$ denotes the complex cohomology degree of a class $\gamma$, that is $\gamma \in H^{2 \deg(\gamma)}$.
		(Otherwise the invariants below will vanish, so there is nothing to prove.) We apply the virtual localization formula to the series
		\[
		Z^{S \times \p^1/S_{\infty}}_{(\beta,n)}\left( \prod_{i} \tau_{k_i}(\omega \alpha_i) \middle| \, \lambda \,  \right)
		\]
		with respect to the scaling action on the $\p^1$,
		where we lift all the point classes $\omega \in H^2(\p^1)$ to the equivariant class $[\mathbf{0}] \in H^2_{\BC^{\ast}}(\p^1)$.
		The formula has contributions from the extremal, the mixed and the pure-rubber components.
		
		By the discussion in Section~\ref{subsec:mixed pure rubber}
		the contribution from the mixed components vanishes. 
		The contribution from the pure-rubber components is
		\[
		Z^{S \times \p^1 / S_{0, \infty}, \sim}_{(\beta,n)}
		\left( \frac{1}{t - c_1(\CL_0)} \middle| \lambda, \gamma \right)
		\]
		where $\sim$ stands for rubber, and the second relative insertion is
		\[ \gamma = \prod_{i=1}^{\ell} \tau^{S^{[n]}}_{k_i}(\alpha_i) \in H^{\ast}(S^{[n]}). \]
		Since we have the (non-equivariant) degree $\deg(\gamma) + \deg(\lambda) = 2n+1$ by \eqref{deg condition}
		and the rubber space is of (non-equivariant) virtual dimension $2n$, the above integral vanishes. Only the extremal component contributes in the virtual localization.
		
		Moreover, Proposition~\ref{prop:double cosection} shows that
		from the extremal component
		only the 1-step component can contribute.
		
		We analyze now the contribution from the 1-step component.
		By equation \eqref{EXTREMECONTRIBUTION} in Section~\ref{subsec:single step} and Remark~\ref{rmk:1 step component in rank 1}
%		and since we are in rank $1$ (Section~\ref{subsec:cap rank 1})
		one finds that:
		\begin{gather*}
			Z^{S \times \p^1/S_{\infty}}_{(\beta,n)}\left( \prod_{i} \tau_{k_i}(\omega \alpha_i) \middle| \, \lambda \,  \right)
			=
			\sum_{m} (-p)^m
			\int_{[  
				\Hilb_{n,\beta,m}(S \times \p^1/S_{\infty})
				]^{\text{vir}} }
			\prod_{i} \tau_{k_i}(\omega \alpha_i) \ev_{\infty}^{\ast}(\lambda) \\
			=
			\sum_{m} (-p)^m
			\sum_{r|m}
			(-1)^{(r-1) \frac{m}{r}}
			\int_{ [ S_{\beta/r}^{[n_1, n]} ]^{\vir} }
			e_{\BC^{\ast}}\left( R \Hom_S(\CK, \CF-\CK)^{\vee} \otimes \Ft^{-r} \right) 
			\prod_i
			\tau_{k_i}( [ \mathbf{0}] \alpha_i) \cdot \pi_2^{\ast}(\lambda)
		\end{gather*}
		where $n_1 = n + \frac{m}{r} + \frac{1}{2} \beta^2/r^2$.
		We analyze the descendent insertion in the next lemma.
		\begin{lemma} \label{lemma:descedents}
			Under the identification $M_r^{1\text{-step}} \cong S_{\beta/r}^{[n_1, n]}$ we have that
			\[ \tau_{k}( [ \mathbf{0}] \alpha)|_{M_r^{1\text{-step}}} = \sum_{d \geq 0} \sigma_{d}(\alpha,\beta/r) ( r t )^{d} \]
			where $\sigma_{d}(\alpha,\beta')$ is a universal (i.e. independent of $\alpha,\beta'$) polynomial of complex cohomological degree $\deg \sigma_{d}(\alpha,\beta') = \deg \tau_{k}(\omega \alpha) - d$
			in the following variables:
			% $\tau_{j_1}^{S^{[n_1]}}(\alpha \beta^{\prime s})$,
			%$\tau_{j_2}^{S^{[n]}}(\alpha)$ and $z = c_1(\CO_{\p}(1))$.
			\[
			\tau_{j_1}^{S^{[n_1]}}(\alpha \beta^{\prime s}), \quad 
			\tau_{j_2}^{S^{[n]}}(\beta'), \quad z = c_1(\CO_{\p}(1)),
			\quad \int_S \alpha \beta^{\prime s}, \quad s \in \{ 0,1,2 \}, \quad j_1, j_2 \geq 0.
			\]
		\end{lemma}
		\begin{proof}[Proof of Lemma~\ref{lemma:descedents}]
			Let $\CJ$ denote the universal ideal sheaf over $M_r^{1\text{-step}}$. It sits in an exact sequence
			\[ 0 \to \pi_S^{\ast}(\CI_2) \otimes \CO(-r) \to \CJ \to \iota_{r \ast} \pi_S^{\ast}( \CI_1(-\beta/r) \otimes \CO_{\p}(-1) ) \to 0, \]
			where $\iota_r : S \times \Spec(k[x]/x^r) \to S \times \p^1$ is the inclusion.
			Since we have the exact sequence 
			\[ 0 \to \CO_{\p^1}(-r) \to \CO_{S \times \p^1} \to \iota_{r \ast} \CO \to 0 \]
			we obtain $\ch( \iota_{r \ast}(1) )|_{S_{0}} = 1 - e^{- r t}$ and therefore
			\[ \ch(\CJ) \cup [ \mathbf{0}] = \Big( \ch(\CI_2) \otimes e^{-r t} + (1-e^{-rt}) \ch( \CI_1 ) e^{-\beta/r} e^{-z} \Big) \cup [ \mathbf{0}]. \]
			Cupping with $\pi_S^{\ast}(\alpha)$ and pushing forward to $M_r^{1\text{-step}}$ we find
			\begin{align*}
				\tau_{k}( [ \mathbf{0}] \alpha)|_{M_r^{1\text{-step}}}
				& = -\pi_{\ast}( \ch_{2+k}(\CJ) \cup [ \mathbf{0}] \pi_S^{\ast}(\alpha)) \\
				& = 
				- \left[ \pi^{(2)}_{\ast}( \ch(\CI_2) \pi_S^{\ast}(\alpha)) e^{-r t} + (1-e^{-rt}) \pi^{(1)}_{\ast}\left( \ch( \CI_1 ) \pi_S^{\ast}( e^{-\beta/r} \alpha ) \right) e^{-z} \right]_{\deg(\tau_k(\alpha \omega))} \\
				& = 
				\Bigg[ \left( -({\scriptstyle \int_{S}} \alpha) 1 + \sum_{j_2 \geq 0} \tau_{j_2}^{S^{[n]}}(\alpha) \right) e^{-r t}  \\
				%\pi^{(2)}_{\ast}( \ch(\CI_2) \pi_S^{\ast}(\alpha)) e^{-r t} + 
				& \quad + (1-e^{-rt}) 
				\left( -({\scriptstyle \int_{S}} e^{-\beta/r} \alpha) 1 +
				\sum_{j_1 \geq 0} \tau_{j_1}^{S^{[n_1]}}(\alpha e^{-\beta/r}) \right) e^{-z} \Bigg]_{\deg(\tau_k(\alpha \omega))}
			\end{align*}
			where $\pi^{(i)} : S^{[n_i]} \times S \to S^{[n_i]}$ are the projections.
			This concludes the claim.
			%The claim now follows by taking Chern characters and pushing forward.
			%(Note that $\ch( \iota_{r \ast}(1) ) = 1 - e^{- r t}$ when restricted to $S \times \BA^1$.)
		\end{proof}
		
		We continue the proof of Theorem~\ref{thm:mc for real}.
		We have 
		$\rk R \Hom_S(\CK, \CF-\CK)^{\vee} = (\beta/r)^2 + m/r$, and hence can write
		\begin{equation} \label{e expnasion}
			e_{\BC^{\ast}}( R \Hom_S(\CK, \CF-\CK)^{\vee} \otimes \Ft^{-r})
			=
			\sum_{j} c_{(\beta/r)^2 + m/r+j}( R \Hom_S(\CK, \CF-\CK) ) (-rt)^{-j}.
		\end{equation}
		Moreover, $[ S_{\beta/r}^{[n_1, n]} ]^{\vir}$
		is of dimension $2n+1+m/r + \beta^2/r^2$.
		By Lemma~\ref{lemma:descedents} and since by the degree condition \eqref{deg condition} 
		the term $(r t)^{-j}$ coming from the expansion \eqref{e expnasion}
		cancels with the term $(r t)^{d_1 + \ldots + d_r}$ coming from the expansion of the $\tau_{k_i}([ \mathbf{0} ] \alpha_i)$,
		we conclude the following equality of non-equivariant integrals:
		\begin{gather*}
			Z^{S \times \p^1/S_{\infty}}_{(\beta,n)}\left( \prod_{i} \tau_{k_i}(\omega \alpha_i) \middle| \, \lambda \,  \right)
			=
			\sum_{m} (-p)^m
			\sum_{r|m}
			(-1)^{(r-1) \frac{m}{r}} \\
			\sum_j \sum_{d_1 + \ldots + d_{\ell} = j} (-1)^j
			\int_{ [ S_{\beta/r}^{[n_1, n]} ]^{\vir} }
			c_{(\beta/r)^2 + m/r+j}( R \Hom_S(\CK, \CF-\CK)^{\vee} )
			\prod_i \sigma_{d_i}(\alpha_i, \beta/r) \cdot  \pi_2^{\ast}(\lambda)
		\end{gather*}
		
		By Theorem~\ref{thm:universality} this integral depends upon $S$, $\beta/r$, $\alpha_i$ and $\lambda = \prod_{i} \Fq_{\lambda_i}(\delta_i)$
		only through the intersection pairings
		of the classes $\beta/r, \alpha_i, \delta_i,1, \pt$.
		Hence we may replace them by the isometric data given by
		\[ \varphi_r( \beta/r ), \varphi_r(\alpha_i), \varphi_r(\delta_i), 1, \pt. \]
		Inserting, and applying the above arguments backwards the above becomes
		\begin{gather*}
			%= \sum_{m} (-p)^m \sum_{r|m} (-1)^{(r-1) \frac{m}{r}} \int_{ [ S_{\varphi_r(\beta/r)}^{[n_1, n]} ]^{\vir} } e\left( R \Hom_S(K, F-K) \right) rod_i \tau_{k_i}(\varphi_r\alpha_i) \cdot \pi_2^{\ast}( \varphi_r(\lambda)) \\
			= \sum_{m} (-p)^m \sum_{r|m} (-1)^{(r-1) \frac{m}{r}} 
			\sum_j \sum_{d_1 + \ldots + d_{\ell} = j} \\
			\int_{ [ S_{\varphi_r(\beta/r)}^{[n_1, n]} ]^{\vir} }
			c_{\varphi_r(\beta/r)^2 + m/r+j}( R \Hom_{S_r}(\CK, \CF-\CK)^{\vee} )
			\prod_i \sigma_{d_i}(\varphi_r(\alpha_i), \varphi_r(\beta/r)) \cdot  \pi_2^{\ast}(\varphi_r(\lambda)) \\
			= 
			\sum_{m} (-p)^m
			\sum_{r|m}
			(-1)^{(r-1) \frac{m}{r}}
			\int_{[  
				\Hilb_{n,\varphi_r(\beta/r),m/r}(S \times \p^1/S_{\infty})
				]^{\text{vir}} }
			\prod_{i} \tau_{k_i}(\omega \varphi_r(\alpha_i)) \ev_{\infty}^{\ast}(\varphi_r(\lambda)) \\
			= 
			\sum_{r|m} 
			Z^{S \times \p^1/S_{\infty}}_{(\varphi_r(\beta/r),n)}\left( \prod_{i} \tau_{k_i}(\omega \varphi_r(\alpha_i)) \middle| \, \varphi_r(\lambda) \,  \right)(p^{r})
		\end{gather*}
		which was what we wanted to prove.
	\end{proof}

	\section{Proofs of the remaining main results}
	\label{sec:Proofs}
%	\subsection{Proof of Theorems~\ref{thm:Wall  cross rank 1} and \ref{thm:wallcrossing higher rank}}
%	This follows directly from Theorem~\ref{thm:wallcrossing term for real} and the main results of \cite{N1,N2}. \qed
	
	\subsection{Proof of Theorem~\ref{thm:Correction term}}
The first part (the independence from the divisibility) follows by Proposition~\ref{prop:div independence for Q}.
We hence need to evaluate $Q_{n,h,m}$.	
	Recall the generating series $\DT_n(p,q)$ and $\HH_n(p,q)$ from \eqref{DTnHHn},
	and the series 
	\[ \mathsf{Q}_n(p,q) = \sum_{h \geq 0} \sum_{m \in \BZ} \mathsf{Q}_{n,h,m} q^{h-1} p^m. \]
	By Theorem~\ref{thm:Correction term} we have for all $n$ the equality
	\begin{equation} \DT_n(p,q) = \HH_n(p,q) - \chi(S^{[n]}) \mathsf{Q}_n(p,q). \label{wallcrossing} \end{equation}
	
	Since $S^{[0]} = pt$ and $S^{[1]} = S$, by the Yau-Zaslow formula we have
	\[ \HH_0 = 0, \quad \HH_1 = -2 \frac{E_2(q)}{\Delta(q)}. \]
	On the other hand, by the Katz-Klemm-Vafa formula \cite{MPT} for $n=0$,
	and by \cite{Br} and \cite{K3xP1} (see also \cite{HAE}) for $n=1$ we have that
	\[
	\DT_0 = - \frac{1}{\Theta^2 \Delta}, \quad \DT_1 = -24 \frac{\wp(p,q)}{\Delta(q)}.
	\]
	where $\wp(p,q)$ is the Weierstra{\ss} elliptic function (see \cite[Sec.2]{K3xE}).
	
	By Proposition~\ref{prop:Qmultiplicativity}
	we have $\mathsf{Q}_n(p,q) = F_1^n F_2$ for some $F_1, F_2$. So from case $n=0$ we conclude that:
	\[ F_2 = \frac{1}{\Theta^2 \Delta}. \]
	For the $n=1$ term we conclude
	\[ F_1 = \frac{1}{24 F_2}( \HH_1 - \DT_1 )
	= \Theta^2 \cdot ( -\frac{1}{12} E_2 + \wp ) = \mathbf{G}(p,q)
	\]
	where we used that $\left( \frac{p}{dp} \right)^2 \log(\Theta(p,q)) = -\wp(p,q) + \frac{1}{12} E_2(q)$,
	see \cite[Equation (11)]{K3xE}.\footnote{Note that $F(z,\tau)$ in \cite{K3xE} corresponds to $-i \Theta(p,q)$. The variable convention is the same.} \qed 
	% the last equality is \cite[Equation (11)]{K3xE}. \qed
	
	\begin{rmk}
If one had the GW/DT correspondence for the cap geometry $(S \times \p^1)/S_{\infty}$ the above computation of $F_1$ would also follow from the results of \cite{K3xP1}.\footnote{The GW/DT correspondence for $(S \times \p^1)/S_{\infty}$ was recently proven in \cite{Marked}.}
%The correspondence will be taken up in follow-up work. \qed
	\end{rmk}

	\subsection{Proof of Theorem~\ref{thm: mc K3xE}}
	Define the series
	\[ Z^{S \times E/E}_{(\beta,n)}
	=
	\sum_{m \in \BZ} \DT^{S \times E}_{m, (\beta,n)} (-p)^m.
	\]
	Choose a class $D \in H^2(S,\BZ)$ such that $D \cdot \beta \neq 0$.
	By \cite{ObReduced} we have that
	\[
	Z^{S \times E/E}_{(\beta,n)} = \frac{1}{\beta \cdot D}
	Z^{S \times E}_{(\beta,n)}( \tau_0(\omega D)).
	\]
	By Theorem~\ref{thm:mc for real} we hence have that:
	\begin{align*}
		Z^{S \times E/E}_{(\beta,n)}
		& = 
		\frac{1}{\beta \cdot D}
		Z^{S \times E}_{(\beta,n)}( \tau_0(\omega D)) \\
		& = \frac{1}{D \cdot \beta} \sum_{r | \beta}
		Z^{S \times E}_{(\varphi_r(\beta/r), n)}( \tau_0( \omega \varphi_r(D) ) )(p^r) \\
		& = \sum_{r | \beta}
		\frac{ \varphi_r(D) \cdot \varphi_r(\beta/r)}{D \cdot \beta}
		Z^{S \times E/E}_{(\varphi_r(\beta/r), n)}(p^r) \\
		& = \sum_{r | \beta}  \frac{1}{r} Z^{S \times E/E}_{(\varphi_r(\beta/r), n)}(p^r).
	\end{align*}
	This implies the claim by taking coefficients. \qed

	\subsection{Proof of Theorem~\ref{thm: mc Hilb}}
	Since the multiple cover formula is compatible with the divisor equation and restriction of Gromov-Witten classes to boundary components,
	it is enough to consider the case $N=3$ and $\alpha = 1$.
	By \cite[Cor. 4.2]{N2} these primary invariants are identical to the
	DT invariants of the relative geometry $S \times \p^1 / S_{0,1,\infty}$.
	A small calculation shows that the multiple cover formula given in Theorem~\ref{thm:mc for real} 
	implies the form of the multiple cover formula given in Theorem~\ref{thm: mc Hilb}. \qed

	Mathemathisches Institut, Universit\"at Bonn
	
	georgo@math.uni-bonn.de \\
	
\end{document}